\documentclass[twocolumn]{article}
\usepackage{amsfonts,amsmath,amsthm,latexsym,hyperref} % label
\allowdisplaybreaks

\newcounter{defnum}
\newtheorem{definition}{Definition}[defnum]

\newcounter{lemnum}
\newtheorem{lemma}{Lemma}[lemnum]

\newcounter{remnum}
\newtheorem{remark}{Remark}[remnum]

\newcounter{satznum}
\newtheorem{theorem}{Theorem}[satznum]

\newenvironment{example}
 {\begin{trivlist}\item[]{\bf Example.}}
 {\end{trivlist}}

{\makeatletter
\gdef\me{\mathbb{E}} % mean, expectation
\gdef\nz{\mathbb{N}} % positive integers
\gdef\pr{\mathbb{P}} % probability
\gdef\rz{\mathbb{R}} % real numbers
\gdef\gz{\mathbb{Z}} % integers
}
\newcounter{todocounter}

\makeatletter
\def\@MRExtract#1 #2!{#1}
\newcommand{\MR}[1]{% we need to strip the "(...)"
  \xdef\@MRSTRIP{\@MRExtract#1 !}
  \href{http://www.ams.org/mathscinet-getitem?mr=\@MRSTRIP}{MR\@MRSTRIP}}
\makeatother
\begin{document}
\twocolumn[ % inside this optional argument goes one-column text
\begin{@twocolumnfalse}
\centering
   \section*{ON MULTI-TYPE CANNINGS MODELS AND MULTI-TYPE EXCHANGEABLE COALESCENTS}
   {\sc Martin M\"ohle}\footnotemark
%\footnote{Mathematisches Institut, Eberhard Karls Universit\"at T\"ubingen, Auf der Morgenstelle 10, 72076 T\"ubingen, Germany, E-mail address: martin.moehle@uni-tuebingen.de}
%
\begin{center}
   Date: \today\\
%   Draft Version\\
%   (in preparation for Theoret. Popul. Biol.)\\
%   Submission deadline: May 31, 2023
\end{center}
\begin{abstract}
   A multi-type neutral Cannings population model with mutation and fixed subpopulation sizes is analyzed. Under appropriate conditions, as all subpopulation sizes tend to infinity, the ancestral process, properly time-scaled, converges to a multi-type exchangeable coalescent with mutation sharing the exchangeability and consistency property. The proof gains from coalescent theory for single-type Cannings models and from decompositions into reproductive and mutational parts. The second part deals with a different but closely related multi-type Cannings model with mutation and fixed total population size but stochastically varying subpopulation sizes. The latter model is analyzed forward and backward in time with an emphasis on its behaviour as the total population size tends to infinity. Forward in time, multitype limiting branching processes arise for large population size. Its backward structure and related open problems are briefly discussed.

   \vspace{2mm}

   \noindent Keywords: Consistency; exchangeability; mutation; multi-type branching process; multi-type Cannings model; multi-type coalescent

   \vspace{2mm}

   \noindent 2020 Mathematics Subject Classification:
            Primary 60J90; % Coalescent processes
            60J10 % Markov chains (discrete-time Markov processes on discrete state space)
            Secondary 92D15; % Problems related to evolution
            92D25 % Population dynamics (general)
\end{abstract}

\vspace{5mm}

\end{@twocolumnfalse}
] % end of one-column text

\footnotetext{Mathematisches Institut, Eberhard Karls Universit\"at T\"ubingen, Auf der Morgenstelle 10, 72076 T\"ubingen, Germany, E-mail address: martin.moehle@uni-tuebingen.de}

% twocolumn text here

\section{Introduction}

Multi-type models play an important role in mathematical population genetics and evolutionary game theory. Many of the early multi-type population models belong to the class of multi-type branching processes (Athreya and Ney \cite{AthreyaNey1972}, Harris \cite{Harris1963}) but also different multi-type population models have been studied in the early literature, for example based on Poissonian or renewal inputs (Port \cite{Port1968}). In all these models, each individual has a certain type which determines (the distribution of) the number of offspring of that individual of any type. In applications, individuals might be objects like genes in biology, particles in physics or balls in combinatorial urn models. More recent and more advanced models even allow for genetic forces like selection and recombination. The reader is referred to the book of Ewens \cite{Ewens2004} for an overview on models in mathematical population genetics. We also refer exemplary to the works of Etheridge \cite{Etheridge2011}, Etheridge and Griffiths \cite{EtheridgeGriffiths2009}, Etheridge, Griffiths and Taylor \cite{EtheridgeGriffithsTaylor2010} and Griffiths \cite{Griffiths1981} to point to (multi-type) models including biological forces such as selection and recombination. We are however interested in the neutral case where neither selection nor recombination acts in the population. To avoid technical difficulties it is assumed that the space $E$ of possible types is finite or countable infinite. Our intension is to study some neutral multi-type population models with non-overlapping generations and mutation in the spirit of Cannings \cite{Cannings1974,Cannings1975,Cannings1976}. In the first model studied in Section \ref{first}, all subpopulation sizes are constant, whereas in the second model studied in Section \ref{second}, the total population size is constant.

Before we turn to these multi-type models let us mention some fundamental and well-known results concerning single-type Cannings models. Cannings studied a population model with non-overlapping generations and a fixed number $N\in\nz:=\{1,2,\ldots\}$ of individuals in each generation $r\in\gz:=\{\ldots,-1,0,1,\ldots\}$. Each individual $i\in[N]:=\{1,\ldots,N\}$ alive in generation $r\in\gz$ produces a random number $\nu_{i,N}^{(r)}$ of offspring and these offspring form the next generation. It is assumed that, for each generation $r$, the offspring sizes $\nu_{i,N}^{(r)}$, $i\in[N]$, are exchangeable and that the offspring vectors $\nu_N^{(r)}:=(\nu_{1,N}^{(r)},\ldots,\nu_{N,N}^{(r)})$, $r\in\gz$, are independent and identically distributed (iid) over the generations. Since the total population size is assumed to be constant equal to $N$, the relation $\nu_{1,N}^{(r)}+\cdots+\nu_{N,N}^{(r)}=N$ holds for all generations $r$. We define $\nu_{i,N}:=\nu_{i,N}^{(0)}$ for convenience for all $N\in\nz$ and $i\in[N]$.

We are mainly interested in the ancestral structure of these models. For $n\in\nz$ let ${\cal P}_n$ denote the set of partitions of $[n]:=\{1,\ldots,n\}$. Take a sample of $n\in[N]$ individuals from generation $0$. For each $r\in\nz_0:=\{0,1,\ldots\}$ define a random partition $\Pi_r:=\Pi_r^{(n,N)}$ by the property that $i,j\in[n]$ belong to the same block of $\Pi_r$ if and only if the individuals $i$ and $j$ share a common ancestor $r$ generations backward in time. It is well known that $(\Pi_r)_{r\in\nz_0}$ is a time-homogeneous Markov chain with state space ${\cal P}_n$ and initial state $\{\{1\},\ldots,\{n\}\}$. For $x\in\rz$ define $(x)_0:=1$ and $(x)_j:=x(x-1)\cdots(x-j+1)$ for all $j\in\nz$. Transitions from $\pi\in{\cal P}_n$ to $\pi'\in{\cal P}_n$ are only possible with positive probability if each block of $\pi'$ is a union of some blocks of $\pi$. In this case the transition probability $p_{\pi\pi'}:=\pr(\Pi_r=\pi'\,|\,\Pi_{r-1}=\pi)$ is given by (see \cite[Eq.~(3)]{MoehleSagitov2001})
\begin{equation} \label{p}
p_{\pi\pi'}\ =\ \frac{(N)_j}{(N)_i}\me((\nu_{1,N})_{i_1}\cdots(\nu_{j,N})_{i_j}),
\end{equation}
where $i$  and $j$ denote the number of blocks of $\pi$ and $\pi'$ respectively and $i_1,\ldots,i_j$ are the group sizes of merging blocks of $\pi$. Note that $i_1+\cdots+i_j=i$. The functions $\Phi_j^{(N)}:\{(k_1,\ldots,k_j)\in\nz^j:k_1+\cdots+k_j\le N\}\to[0,1]$, $N\in\nz$, $j\in[N]$, defined via
\begin{equation} \label{phij}
   \Phi_j^{(N)}(k_1,\ldots,k_j)\ :=\ \frac{(N)_j}{(N)_k}\me((\nu_{1,N})_{k_1}\cdots(\nu_{j,N})_{k_j})
\end{equation}
for all $k_1,\ldots,k_j\in\nz$ with $k:=k_1+\cdots+k_j\le N$, therefore play a key role in the analysis of the ancestral structure of Cannings population models. In particular, for $N\in\nz\setminus\{1\}$,
\begin{equation} \label{coalprob}
   c_N\ :=\ \Phi_1^{(N)}(2)\ =\ \frac{\me((\nu_{1,N})_2)}{N-1}\ =\ \frac{{\rm Var}(\nu_{1,N})}{N-1}
\end{equation}
is the probability that two individuals, randomly sampled from some generation, share a common ancestor one generation backward in time. This probability, called the \emph{coalescence probability}, is of fundamental interest in coalescent theory, since $1/c_N$ is the time-scale in order to obtain convergence of the ancestral process as the total population size $N$ tends to infinity. The quantity $N_e:=1/c_N$ is called the \emph{effective population size} of the Cannings model. The transition functions $\Phi_j^{(N)}$, $j\in\nz$, satisfy the consistency relation (see \cite[Eqs.~(3) and (4)]{Moehle2002})
\begin{eqnarray}
   &   & \hspace{-10mm}\Phi_j^{(N)}(k_1,\ldots,k_j)\ =\ \Phi_{j+1}^{(N)}(k_1,\ldots,k_j,1)\nonumber\\
   &   & +\sum_{i=1}^j\Phi_j^{(N)}(k_1,\ldots,k_{i-1},k_i+1,k_{i+1},\ldots,k_j),
         \label{consistency}
\end{eqnarray}
$k_1,\ldots,k_j\in\nz$ with $k_1+\cdots+k_j<N$, and these functions are monotone in the sense that
\begin{equation} \label{monotone}
   \Phi_j^{(N)}(k_1,\ldots,k_j)\ \le\ \Phi_\ell^{(N)}(m_1,\ldots,m_\ell)
\end{equation}
for all $1\le\ell\le j\le N$ and $k_1,\ldots,k_j,m_1,\ldots,m_\ell\in\nz$ with $k_1\ge m_1,\ldots,k_\ell\ge m_\ell$ and $k_1+\cdots+k_j\le N$. Note that (\ref{monotone}) follows from (\ref{consistency}) by induction on the difference $j-\ell$. The consistency (\ref{consistency}) and the monotonicity (\ref{monotone}) play a key role (see, for example, \cite{MoehleSagitov2001}) in the analysis of the ancestral structure of Cannings models as $N\to\infty$, which is the main reason why we mention these two properties already in the introduction.

The paper is organized as follows. In Section \ref{first}, a multi-type Cannings model with fixed subpopulation sizes is introduced and its ancestral structure is analyzed in detail. The main convergence result (Theorem \ref{main}) is provided and verified in Subsection \ref{limiting}. Multi-type exchangeable coalescents with mutation and their block counting processes are considered in Subsections \ref{multicoal} and \ref{blockcounting} respectively.

A closely related multi-type Cannings model with fixed total population size but variable subpopulation sizes is studied in Section \ref{second}. Subsection \ref{forwardstructure} studies its structure forward in time. A limiting multi-type branching process is provided in Subsection \ref{branching}. Some backward results on this model are discussed in Subsection \ref{somebackward}. Details on the particular multi-type Kimura model are deferred to Subsection \ref{kimuramodel}. The paper finishes with a summary and discussion of some open problems in Section \ref{discussion}.

\section{A multi-type Cannings population model with fixed subpopulation sizes} \label{first}

It is assumed that the type space $E$ is finite or countable infinite and that the number of individuals of type $k\in E$ does not change over the generations and is hence equal to some given constant $N_k\in\nz$. Reproduction within each subpopulation of type $k\in E$ is assumed to take place according to a neutral Cannings population model (as described in the introduction) with population size $N_k$ and offspring sizes $\nu_{i,N_k,k}$, $i\in\{1,\ldots,N_k\}$. The additional index $k$ indicates that (the distribution of) the number $\nu_{i,N_k,k}$ of offspring of the $i$-th individual in subpopulation $k$ is allowed to explicitly depend on $k$. Offspring sizes in different subpopulations or in different generations are assumed to be independent.

In each generation, a mutation step follows the reproduction step. A given number $N_{k\ell}$ of the $N_k$ children born in subpopulation $k\in E$ mutate to type $\ell\in E$ with $\ell\ne k$. Note that $\sum_{\ell\ne k}N_{k\ell}\in\{0,\ldots,N_k\}$ and that $N_{kk}:=N_k-\sum_{\ell\ne k}N_{k\ell}$ of the children born in subpopulation $k$ do not mutate and, hence, keep their type $k$. Since the size of each subpopulation is assumed to be constant, the conservation equalities
\begin{equation} \label{conservative}
   \sum_{\ell\ne k}N_{k\ell}\ =\ \sum_{\ell\ne k}N_{\ell k},\qquad k\in E,
\end{equation}
are required. Particular models of this form with subpopulation sizes $N_k=2c_kN$ for some $c_k,N\in\nz$, and Wright--Fisher reproduction in each subpopulation have been studied by Notohara \cite{Notohara1990} and (Wilkinson--)Herbots \cite{Herbots1994,Herbots1997,WilkinsonHerbots1998}. In these works types are interpreted as colonies and mutation as migration between these colonies.

Each children born in subpopulation $k\in E$ has probability $N_{k\ell}/N_k$ to mutate to type $\ell\in E$ with $\ell\ne k$. However, since the subpopulation sizes are constant, the individuals do not mutate independently. For any given sample of $n_k$ ($\in[N_k]$) children taken from the $N_k$ children born in subpopulation $k\in E$, and for given integers $n_{k\ell}\in\nz_0$, $\ell\ne k$, with $\sum_{\ell\ne k}n_{k\ell}\le n_k$, the probability that, for all $\ell\in E$ with $\ell\ne k$, $n_{k\ell}$ of these $n_k$ children mutate to type $\ell$, is given by the multi-hypergeometric expression
\begin{equation} \label{hyper}
   \frac{\prod_{\ell\in E}\binom{N_{k\ell}}{n_{k\ell}}}{\binom{N_k}{n_k}},
\end{equation}
where $N_{kk}:=N_k-\sum_{\ell\ne k}N_{k\ell}$ and $n_{kk}:=n_k-\sum_{\ell\ne k} n_{k\ell}$ is the number of children in the sample which do not mutate.

Before the ancestral structure of this model is described, the notion of typed (or labelled) partitions is introduced. Let $n\in\nz$. Each partition of $[n]$ can be written as $\{B_1,\ldots,B_j\}$, where $B_1,\ldots,B_j$ are the (non-empty) blocks of the partition. Note that $\bigcup_{i=1}^jB_i=[n]$. We additionally equip each block $B_i$, $i\in[j]$, with a type $k_i\in E$ and call the set ${\cal P}_{n,E}$ consisting of all $\pi:=\{(B_1,k_1),\ldots,(B_j,k_j)\}$ satisfying $\{B_1,\ldots,B_j\}\in{\cal P}_n$ and $k_1,\ldots,k_j\in E$ the space of typed (or labelled) partitions of $[n]$. Each typed partition of $[n]$ can be viewed as a usual partition of $[n]$ with the additional property that each block of the partition is painted with some `color' taken from the space $E$ of possible `colors'. We also call
$B_i$ a $k_i$-block of $\pi$, $i\in[j]$.

Take a sample of $n\in[N]$ individuals from generation $0$, label them (in some arbitrary order) from $1$ to $n$, and let $k_1,\ldots,k_n\in E$ denote the types of these individuals. The ancestry of the individuals in the sample can be traced back as follows. For $r\in\nz_0$ define a random typed partition $\Pi_r=\Pi_r^{(n,(N_k)_{k\in E})}$ of $[n]$ by the property that $i,j\in[n]$ belong to the same $k$-block of $\Pi_r$ if and only if the individuals $i$ and $j$ share a common ancestor $r$ generations backward in time and this ancestor has type $k$. The process $(\Pi_r)_{r\in\nz_0}$ is called a multi-type ancestral process, sometimes also a multi-type backward process or a multi-type discrete coalescent process. It is readily checked that $(\Pi_r)_{r\in\nz_0}$ is a time-homogeneous Markov chain with state space ${\cal P}_{n,E}$ and initial state $\{(\{1\},k_1),\ldots,(\{n\},k_n)\}$, where $k_i\in E$ denotes the type of individual $i\in[n]$. Let $p_{\pi\pi'}:=\pr(\Pi_r=\pi'\,|\,\Pi_{r-1}=\pi)$, $\pi,\pi'\in{\cal P}_{n,E}$, denote the transition probabilities. From the two-step definition of the model it follows that the transition matrix $P:=(p_{\pi\pi'})_{\pi,\pi'\in{\cal P}_{n,E}}$ has the product form
\begin{equation} \label{P}
   P\ =\ P^{\rm mut}P^{\rm rep},
\end{equation}
where $P^{\rm mut}$ denotes the transition matrix of the ancestral process for the model without reproduction, i.e. for the model with $\nu_{i,N_k,k}=1$ almost surely for all $k\in E$ and all $i\in\{1,\ldots,N_k\}$, and $P^{\rm rep}$ denotes the transition matrix of the ancestral process for the model without mutation, i.e. for the model with $N_{k\ell}=0$ for all $k,\ell\in E$ with $k\ne\ell$.

For the model without mutation, transitions from $\pi\in{\cal P}_{n,E}$ to $\pi'\in{\cal P}_{n,E}$ are only possible with positive probability if each $k$-block of $\pi'$ is a union of some $k$-blocks of $\pi$. In this case, since offspring numbers in different subpopulations are independent, it follows that
\begin{equation} \label{p_rep}
   p_{\pi\pi'}^{\rm rep}\ =\ \prod_{k\in E} \Phi_{j_k}^{(k,N_k)}(i_{k,1},\ldots,i_{k,j_k}),
\end{equation}
where $i_k$ and $j_k$ are the number of $k$-blocks of $\pi$ and $\pi'$ respectively, $i_{k,1},\ldots,i_{k,j_k}$ are the group sizes of merging $k$-blocks of $\pi$, and
\begin{eqnarray}
   &   & \hspace{-15mm}\Phi_{j_k}^{(k,N_k)}(i_{k,1},\ldots,i_{k,j_k})\nonumber\\
& := & \frac{(N_k)_{j_k}}{(N_k)_{i_k}}
\me\big((\nu_{1,N_k,k})_{i_{k,1}}\cdots(\nu_{j_k,N_k,k})_{i_{k,j_k}}\big).
\end{eqnarray}
Note that $i_{k,1}+\cdots+i_{k,j_k}=i_k$.

For $k,\ell\in E$, the backward mutation probability $m_{k\ell}$, which is by definition the proportion of the individuals in subpopulation $k$ after the mutation step, who where born in subpopulation $\ell$, is
\begin{equation} \label{backwardrate}
   m_{k\ell}\ :=\ \frac{N_{\ell k}}{N_k}%\ =\ \frac{c_\ell}{c_k}q_{\ell k},
   \qquad k,\ell\in E.
\end{equation}
where $N_{kk}:=N_k-\sum_{\ell\ne k}N_{\ell k}$. Note that $m_{kk}$ is the proportion of the individuals in subpopulation $k$, who did not undergo a mutation during the mutation step.

For the model without reproduction, the entries of the matrix $P^{\rm mut}$ are obtained as follows. Let $\pi,\pi'\in{\cal P}_{n,E}$. A mutational transition from $\pi$ to $\pi'$ backward in time is only possible with positive probability if $\pi'$ has the same blocks as $\pi$. For $k,\ell\in E$ let $n_{k\ell}$ denote the number of blocks being a $k$-block of $\pi$ and a $\ell$-block of $\pi'$. Then,
\begin{equation}
   p_{\pi\pi'}^{\rm mut}
   \ =\ \prod_{k\in E}\frac{\prod_{\ell\in E}(N_{\ell k})_{n_{k\ell}}}{(N_k)_{n_k}},
   \label{p_mut_entry}
\end{equation}
where $n_k:=\sum_{\ell\in E}n_{k\ell}$ denotes the number of $k$-blocks of $\pi$. For example, if $\pi=\{([n],k)\}$ and $\pi'=\{([n],\ell)\}$ for some $k,\ell\in E$, then (\ref{p_mut_entry}) reduces to the backward mutation rate $m_{k\ell}$ defined in (\ref{backwardrate}). If all the $N_{k\ell}$ are sufficiently large, then there is essentially no difference between sampling without replacement and sampling with replacement, leading to the approximation
\begin{equation}
   p_{\pi\pi'}^{\rm mut}
   \ \approx\ \prod_{k,\ell\in E} \bigg(\frac{N_{\ell k}}{N_k}\bigg)^{n_{k\ell}}
   \ =\ \prod_{k,\ell\in E} m_{k\ell}^{n_{k\ell}}.
\end{equation}

\subsection{A limiting multi-type coalescent} \label{limiting}
To avoid technical difficulties it is in this section mainly assumed that the number of types is finite, i.e. $|E|<\infty$. The case of countable infinite type space $E$ is briefly discussed in Remark \ref{infinity} at the end of this section.

We are interested in the behavior of the ancestral process $(\Pi_r^{(n,(N_k)_{k\in E})})_{r\in\nz_0}$ as all subpopulation sizes $N_k$, $k\in E$, become large, that is, as
\begin{equation} \label{min}
   N\ :=\ \min_{k\in E}N_k\ \to\ \infty.
\end{equation}
In order to state a convergence result, a couple of assumptions are imposed, which are described in the following. Let us start with the assumptions concerning the Cannings reproduction models acting in each subpopulation. For $k\in E$ let
\begin{equation} \label{kcoal}
   c_k(N_k)
%\ :=\ \frac{\me((\nu_{N_k,1}^{(k)})_2)}{N_k-1}
   \ :=\ \frac{{\rm Var}(\nu_{1,N_k,k})}{N_k-1}
\end{equation}
denote the coalescence probability of the Cannings model acting in subpopulation $k$ with population size $N_k>1$. It is assumed that $c_k(N_k)>0$ for all sufficiently large $N_k$. Note that $c_k(N_k)=0$ if and only if $\nu_{1,N_k,k}=1$ almost surely. For every subpopulation $k\in E$, it is assumed that all the limits
\begin{equation} \label{repassumption}
   \phi_j^{(k)}(i_1,\ldots,i_j)
   \ :=\ \lim_{N_k\to\infty}\frac{\Phi_j^{(k,N_k)}(i_1,\ldots,i_j)}{c_k(N_k)},
\end{equation}
$j,i_1,\ldots,i_j\in\nz$ with $i_1,\ldots,i_j\ge 2$, exist, where
\[
\Phi_j^{(k,N_k)}(i_1,\ldots,i_j)\ :=\ \frac{(N_k)_j}{(N_k)_i}\me((\nu_{1,N_k,k})_{i_1}\cdots(\nu_{j,N_k,k})_{i_j}).
\]
for all $j\in[N_k]$ and $i_1,\ldots,i_j\in\nz$ with $i:=i_1+\cdots+i_j\le N_k$. The existence of the limits (\ref{repassumption}) is a relatively mild condition, since, by the monotonicity property, $\Phi_j^{(k,N_k)}(i_1,\ldots,i_j)\le \Phi_1^{(k,N_k)}(2)=c_k(N_k)$ for all $i_1,\ldots,i_j\ge 2$ which shows that the fraction on the right-hand side of (\ref{repassumption}) is bounded between $0$ and $1$. Moreover, if the limits (\ref{repassumption}) exist for all $j,i_1,\ldots,i_j\in\nz$ with $i_1,\ldots,i_j\ge 2$, then these limits exist for the wider range of parameters $j,i_1,\ldots,i_j\in\nz$ satisfying $i_1+\cdots+i_j>j$, which follows readily from the consistency relation of the functions $\Phi_j^{(k,N_k)}$, $j\in[N_k]$, by induction on the number of $1$'s among the $i_1,\ldots,i_j$. In this case also the limits
\begin{equation} \label{phijk}
   \phi_j^{(k)}(1,\ldots,1)\ :=\ \lim_{N_k\to\infty}\frac{\Phi_j^{(k,N_k)}(1,\ldots,1)-1}{c_k(N_k)},
\end{equation}
$j\in\nz$,
exist. For each $k\in E$, the consistency relations and the monotonicity property of the functions $\Phi_j^{(k,N_k)}$, $j\in[N_k]$, carry over to the limits $\phi_j^{(k)}$, $j\in\nz$. More precisely, for all $k\in E$ and $j,i_1,\ldots,i_j\in\nz$,
\begin{eqnarray*}
   &   & \hspace{-15mm}\phi_j^{(k)}(i_1,\ldots,i_j)
   \ = \ \phi_{j+1}^{(k)}(i_1,\ldots,i_j,1)\\
   &   & + \sum_{m=1}^j\phi_j^{(k)}(i_1,\ldots,i_{m-1},i_m,i_{m+1},\ldots,i_j)
\end{eqnarray*}
and
\[
\phi_j^{(k)}(i_1,\ldots,i_j)\ \le\ \phi_\ell^{(k)}(m_1,\ldots,m_\ell)
\]
for all $\ell,j\in\nz$ with $\ell\le j$ and $i_1,\ldots,i_j,m_1,\ldots,m_\ell\in\nz$ with $i_1\ge m_1,\ldots,i_\ell\ge m_\ell$.

A more delicate calibration assumption on the coalescence probabilities $c_k(N_k)$, $k\in E$, is needed. Recall that $N:=\min_{k\in E}N_k\in\nz$. It is assumed that there exists a sequence $(c_N)_{N\in\nz}$ of positive real numbers such that, for all $k\in E$,
\begin{equation} \label{dk}
   \frac{c_k(N_k)}{c_N}\ \to\ d_k,\qquad N\to\infty,
\end{equation}
for some constant $d_k\ge 0$.

The assumptions on the mutation parameters are standard. It is assumed that, for all types $k,\ell\in E$ with $k\ne\ell$, the backward mutation rate $m_{k\ell}=m_{k\ell}(N)$ depends on $N$ in such a way that
\begin{equation} \label{mutassumption}
   \frac{m_{k\ell}(N)}{c_N}\ \to\ \rho_{k\ell},\qquad N\to\infty,
\end{equation}
for some constant $\rho_{k\ell}\ge0$.

Under these assumptions, for a given sample size $n\in\nz$, two generator matrices $Q^{\rm rep}$ and $Q^{\rm mut}$ can be defined as follows. Let
\begin{equation} \label{q_rep}
   Q^{\rm rep}\ :=\ (q_{\pi\pi'}^{\rm rep})_{\pi,\pi'\in{\cal P}_{n,E}}
\end{equation}
denote the generator matrix with the following entries. Let $\pi,\pi'\in{\cal P}_{n,E}$ with $\pi\ne\pi'$ and such that each $k$-block of $\pi'$ is a union of some $k$-blocks of $\pi$. Let $i_k$ and $j_k$ denote the number of $k$-blocks of $\pi$ and $\pi'$ respectively, and let $i_{k,1},\ldots,i_{k,j_k}$ denote the group sizes of merging $k$-blocks of $\pi$. Note that $i_{k,1}+\cdots+i_{k,j_k}=i_k$. If there exists exactly one $k\in E$ with $i_k>j_k$, then
\begin{equation} \label{q_rep_entry}
   q_{\pi\pi'}^{\rm rep}\ :=\ d_k
   \phi_{j_k}^{(k)}(i_{k,1},\ldots,i_{k,j_k}).
\end{equation}
All other non-diagonal entries of $Q^{\rm rep}$ are (by definition) equal to $0$.

The second generator matrix
\begin{equation} \label{q_mut}
   Q^{\rm mut}\ :=\ (q_{\pi\pi'}^{\rm mut})_{\pi,\pi'\in{\cal P}_{n,E}}
\end{equation}
is defined as follows. Let $\pi=\{(B_1,k_1),\ldots,(B_j,k_j)\}\in{\cal P}_{n,E}$. If $\pi'$ is identical to $\pi$, except for the fact that only one single block of $\pi'$, say $B_i$, has a type $\ell_i$ different from $k_i$, then
\begin{equation} \label{q_mut_entry}
   q_{\pi\pi'}^{\rm mut}\ :=\ \rho_{k_i\ell_i}.
\end{equation}
All other non-diagonal entries of $Q^{\rm mut}$ are (by definition) equal to $0$. Note that $q_{\pi\pi}^{\rm mut}=-\sum_{i=1}^j\rho_{k_i\ell_i}$.

We now state the main convergence result. Recall that the type space $E$ is assumed to be finite and that $N:=\min_{k\in E}N_k$.

\begin{theorem}[Convergence to multi-type coalescents] \label{main} \ \\
   Assume that the following three assumptions hold.
   \begin{enumerate}
      \item[(i)] Reproduction assumption: For every $k\in E$ the limits (\ref{repassumption}) exist for all $j,i_1,\ldots,i_j\in\nz$ with $i_1,\ldots,i_j\ge 2$.
      \item[(ii)] Calibration assumption: There exist positive real numbers $c_1,c_2,\ldots$ such that, for every $k\in E$, (\ref{dk}) holds for some constant $d_k\ge 0$.
      \item[(iii)] Mutation assumption: For any two types $k,\ell\in E$ with $k\ne\ell$, the backward mutation rate $m_{k\ell}=m_{k\ell}(N)$ depends on $N:=\min_{k\in E}N_k$ in such a way that (\ref{mutassumption}) holds for some constant $\rho_{k\ell}\ge 0$ as $N\to\infty$.
   \end{enumerate}
   Sample $n\in\nz$ individuals from generation $0$, label them randomly from $1$ to $n$, and let $k_1,\ldots,k_n\in E$ denote the types of these individuals. Then, the following statement holds.

   If $c_N\to 0$, then the time-scaled multi-type ancestral process $(\Pi_{\lfloor t/c_N \rfloor}^{(n,(N_k)_{k\in E})})_{t\ge 0}$ converges in $D_{{\cal P}_{n,E}}([0,\infty))$ as $N\to\infty$ to a continuous-time limiting Markov process $(\Pi_t^{(n)})_{t\ge 0}$ with state space ${\cal P}_{n,E}$, initial state $\{(\{1\},k_1),\ldots,(\{n\},k_n)\}$ and infinitesimal generator
   \[
   Q\ =\ Q^{\rm rep}+Q^{\rm mut},
   \]
   where $Q^{\rm rep}$ and $Q^{\rm mut}$ are the matrices (\ref{q_rep}) and (\ref{q_mut}) respectively.
\end{theorem}
\begin{remark} \rm
   Due to the structure of the entries (\ref{q_rep_entry}) of the generator $Q^{\rm rep}$, the continuous-time limiting process $\Pi^{(n)}:=(\Pi_t^{(n)})_{t\ge 0}$ arising in Theorem \ref{main} allows for simultaneous multiple mergers of ancestral lineages in one arbitrary subpopulation but not at the same time in more than one subpopulation. During each mutational transition, the process $\Pi^{(n)}$ only allows for a change of the type of one single block. For $|E|=1$, Theorem \ref{main} essentially reduces to Theorem 2.1 of \cite{MoehleSagitov2001}. A convergence result for the situation when $c_N\to c>0$ is provided in Theorem \ref{main2} below.
\end{remark}
\begin{proof}
   Let $\pi,\pi'\in{\cal P}_{n,E}$ with $\pi\ne\pi'$ and such that each $k$-block of $\pi'$ is a union of some $k$-blocks of $\pi$. Then, by (\ref{p_rep}),
   \[
   \frac{p_{\pi\pi'}^{\rm rep}}{c_N}
   \ =\ \frac{1}{c_N}\prod_{k\in E}\Phi_{j_k}^{(k,N_k)}(i_{k,1},\ldots,i_{k,j_k}).
   \]
   If there exists exactly one $k\in E$ with $i_{k,1}+\cdots+i_{k,j_k}>j_k$, then
   \begin{eqnarray*}
      \frac{p_{\pi\pi'}^{\rm rep}}{c_N}
      & = & \frac{\Phi_{j_k}^{(k,N_k)}(i_{k,1},\ldots,i_{k,j_k})}{c_N}
            \prod_{\ell\ne k}\Phi_{j_\ell}^{(\ell,N_\ell)}(1,\ldots,1)\\
      & \to & d_k\phi_{j_k}^{(k)}(i_{k,1},\ldots,i_{k,j_k})
   \ =\ q_{\pi\pi'}^{\rm rep}
   \end{eqnarray*}
   by (\ref{repassumption}) and the comments thereafter. If $\pi'$ is such that in at least two different subpopulations a true merger event takes place, then $p_{\pi\pi'}^{\rm rep}/c_N\to 0=q_{\pi\pi'}^{\rm rep}$ as $N\to\infty$. Moreover,
   \[
   \frac{1-p_{\pi\pi}^{\rm rep}}{c_N}
   \ =\ \sum_{\pi'\ne\pi}\frac{p_{\pi\pi'}^{\rm rep}}{c_N}
   \ \to\ \sum_{\pi'\ne\pi}q_{\pi\pi'}^{\rm rep}\ =\ -q_{\pi\pi}^{\rm rep}
   \]
   as $N\to\infty$. Thus,
   \[
   P^{\rm rep}\ =\ I+c_NQ^{\rm rep}+o(c_N),\qquad N\to\infty.
   \]
   Let us now turn to $P^{\rm mut}$. Let $\pi=\{(B_1,k_1),\ldots,(B_j,k_j)\}\in{\cal P}_{n,E}$. If $\pi'$ is identical to $\pi$, except for the fact that there exists exactly one single block of $\pi'$, say $B_i$, which has a type $\ell_i$ different from $k_i$, then it follows from (\ref{p_mut_entry}) and the assumption that all the limits (\ref{mutassumption}) exist, that
   \begin{eqnarray*}
      \frac{p_{\pi\pi'}^{\rm mut}}{c_N}
      & \sim & \frac{1}{c_N}\prod_{k,\ell\in E}(m_{k\ell}(N))^{n_{k\ell}}
      \ = \ \frac{m_{k_i\ell_i}(N)}{c_N}\\
      & \to & \rho_{k_i\ell_i}\ =\ q_{\pi\pi'}^{\rm mut},
      \qquad N\to\infty.
   \end{eqnarray*}
   Thus, $P^{\rm mut}=I+c_NQ^{\rm mut}+o(c_N)$ as $N\to\infty$. The transition matrix (\ref{P}) of the ancestral process therefore has the asymptotic expansion
   \begin{eqnarray*}
      P
      & = & P^{\rm mut}P^{\rm rep}\\
      & = & (I+c_NQ^{\rm mut}+o(c_N))(I+c_NQ^{\rm rep}+o(c_N))\\
      & = & I+c_NQ + o(c_N),\qquad N\to\infty,
   \end{eqnarray*}
   where $Q:=Q^{\rm rep}+Q^{\rm mut}$. Let $\|.\|$ denote the matrix norm defined via $\|A\|:=\sup_{\pi}\sum_{\pi'}|a_{\pi\pi'}|$ for all $A=(a_{\pi\pi'})_{\pi,\pi'\in{\cal P}_{n,E}}$. It follows for all $t\ge 0$ that $\|P^{\lfloor t/c_N\rfloor}-(I+c_NQ)^{\lfloor t/c_N\rfloor}\|\le \lfloor t/c_N\rfloor \|P-(I+c_NQ)\|\to 0$ as $N\to\infty$, which shows that $P^{\lfloor t/c_N\rfloor}\sim(I+c_NQ)^{\lfloor t/c_N\rfloor}\to e^{tQ}$ as $N\to\infty$. Thus, the convergence of the one-dimensional distributions is established. The convergence of the finite-dimensional distributions follows by exploiting the Markov property of the involved processes.

   It remains to verify the convergence in $D_{{\cal P}_{n,E}}([0,\infty))$. For $t\ge 0$, $f:{\cal P}_{n,E}\to\rz$ and $\pi\in{\cal P}_{n,E}$ define
   \begin{eqnarray*}
      Sf(\pi)
      & := & \me(f(\Pi_r^{(n,(N_k)_{k\in E})})\,|\,\Pi_{r-1}^{(n,(N_k)_{k\in E})}=\pi)\\
      & = & \sum_{\pi'\in{\cal P}_{n,E}} f(\pi')p_{\pi\pi'}
   \end{eqnarray*}
   and $T_tf(\pi):=\sum_{\pi'\in{\cal P}_{n,E}}f(\pi')(e^{tQ})_{\pi\pi'}$. Note that the operator $S$ depends on $n$ and $(N_k)_{k\in E}$ and that $S^mf(\pi)=\sum_{\pi'\in{\cal P}_{n,E}}f(\pi')(P^m)_{\pi\pi'}$ for $m\in\nz_0$ and that $\{T_t\}_{t\ge 0}$ is a (Feller) semigroup on the space $L$ of all functions $f:{\cal P}_{n,E}\to\rz$ with (conservative) generator $Af(\pi):=\sum_{\pi'\in{\cal P}_{n,E}} f(\pi')q_{\pi\pi'}$. Since $|{\cal P}_{n,E}|<\infty$, it follows for all $t\ge 0$ and $f\in L$ that
   \begin{eqnarray*}
      &   & \hspace{-15mm}\|S^{\lfloor t/c_N\rfloor}f-T_tf\|
      \ := \ \sup_{\pi\in{\cal P}_{n,E}}|S^{\lfloor t/c_N\rfloor}f(\pi)-T_tf(\pi)|\\
      & \le & \sup_{\pi\in{\cal P}_{n,E}} \sum_{\pi'\in{\cal P}_{n,E}} |f(\pi')||(S^{\lfloor t/c_N\rfloor})_{\pi\pi'}-(e^{tQ})_{\pi\pi'}|\\
      & \le & \|f\| \sup_{\pi\in{\cal P}_{n,E}}\sum_{\pi'\in{\cal P}_{n,E}}
            |(P^{\lfloor t/c_N\rfloor})_{\pi\pi'}-(e^{tQ})_{\pi\pi'}|\\
      & = & \|f\|\,\|P^{\lfloor t/c_N\rfloor}-e^{tQ}\|\ \to\ 0,\qquad N\to\infty.
   \end{eqnarray*}
   The convergence in $D_{{\cal P}_{n,E}}([0,\infty))$ to a Markov process $(\Pi_t^{(n)})_{t\ge 0}$ with initial state $\{(\{1\},k_1),\ldots,(\{n\},k_n)\}$ and corresponding (Feller) semigroup $\{T_t\}_{t\ge 0}$ thus follows from Ethier and Kurtz \cite[p.~168, Theorem 2.6]{EthierKurtz1986}, applied with $E$ there replaced by ${\cal P}_{n,E}$, $\varepsilon_N:=c_N$ and $Y_N(r)$ there replaced by $\Pi_r^{{(n,(N_k)_{k\in E})}}$, $r\in\nz_0$.
\end{proof}
\begin{example} (Multi-type Kingman $n$-coalescent)
   Assume that Wright--Fisher reproduction acts in each subpopulation. Then, the coalescence probability in subpopulation $k\in E$ is given by $c_k(N_k)=1/N_k$. For simplicity it is assumed that the population sizes are all equal to $N_k=N\in\nz$. Then the calibration assumption (ii) of Theorem \ref{main} obviously holds with $c_N:=1/N>0$ and $d_k:=1$. Clearly, $c_N\to 0$ as $N\to\infty$. From $\Phi_j^{(k)}(i_1,\ldots,i_j)=(N_k)_j/N_k^i$ for all $j,i_1,\ldots,i_j\in\nz$, where $i:=i_1+\cdots+i_j$, it follows that the reproduction assumption (i) holds with $\phi_1^{(k)}(2)=1$. All other limits in (\ref{repassumption}) are equal to $0$. Thus, Theorem \ref{main} is applicable. The limiting process in Theorem \ref{main} is a multi-type Kingman $n$-coalescent in the sense that single binary mergers of two ancestral lineages of the same type occur with rate $1$. Note that binary mergers at the same time in more than one subpopulation are impossible.
\end{example}
In some cases the sequence $(c_N)_{N\in\nz}$ does not converge to $0$ (as assumed in Theorem \ref{main}) but to a positive constant $c>0$. In this situation the following time-discrete variant of Theorem \ref{main} holds.
\begin{theorem} \label{main2}
   Suppose that the assumptions (i), (ii) and (iii) of Theorem \ref{main} hold. Take a sample of $n\in\nz$ individuals from generation $0$ and let $k_1,\ldots,k_n\in E$ denote their types. Then the following statement holds.

   If $c_N\to c>0$ as $N\to\infty$, then the multi-type ancestral process $(\Pi_r^{(n,(N_k)_{k\in E})})_{r\in\nz_0}$ converges in $D_{{\cal P}_{n,E}}(\nz_0)$ as $N\to\infty$ to a discrete-time limiting Markov process $(\Pi_r^{(n)})_{r\in\nz_0}$ with state space ${\cal P}_{n,E}$, initial state $\{(\{1\},k_1),\ldots,(\{n\},k_n)\}$ and transition matrix
   \[
   A\ :=\ A^{\rm mut}A^{\rm rep},
   \]
   where the two stochastic matrices $A^{\rm mut}=(a_{\pi\pi'}^{\rm mut})_{\pi,\pi'\in{\cal P}_{n,E}}$ and $A^{\rm rep}=(a_{\pi\pi'}^{\rm rep})_{\pi,\pi'\in{\cal P}_{n,E}}$ are defined as follows.

   If $\pi'$ has the same blocks as $\pi$, then
   \[
   a_{\pi\pi'}^{\rm mut}\ :=\ \prod_{k,\ell\in E} (c\rho_{k\ell})^{n_{k\ell}},
   \]
   where $n_{k\ell}$ denotes the number of blocks being a $k$-block of $\pi$ and an $\ell$-block of $\pi'$. All other entries of $A^{\rm mut}$ are (by definition) equal to $0$.

   Let $\pi,\pi'\in{\cal P}_{n,E}$ with $\pi\ne\pi'$ and such that each $k$-block of $\pi'$ is a union of some $k$-blocks of $\pi$. Let $i_k$ and
   $j_k$ denote the number of $k$-blocks of $\pi$ and $\pi'$ respectively, and let $i_{k,1},\ldots,i_{k,j_k}$ denote the group sizes of merging $k$-blocks of $\pi$. Then,
   \[
   a_{\pi\pi'}^{\rm rep}\ :=\ \prod_{k\in E} \big(cd_k\phi_{j_k}^{(k)}(i_{k,1},\ldots,i_{k,j_k})\big).
   \]
   All other non-diagonal entries of $A^{\rm rep}$ are (by definition) equal to $0$.
\end{theorem}
\begin{remark} \rm
   In contrast to the continuous-time limiting process in Theorem \ref{main}, the discrete-time limiting process $\Pi^{(n)}:=(\Pi_r^{(n)})_{r\in\nz_0}$ arising in Theorem \ref{main2} allows for simultaneous multiple mergers of ancestral lineages at the same time in more than one subpopulation. During each transition, the process $\Pi^{(n)}$ also allows for a change of the type of more than one single block.
\end{remark}
\begin{proof}
   If $\pi'$ has the same blocks as $\pi$, then, by (\ref{p_mut_entry}),
   \begin{eqnarray*}
      p_{\pi\pi'}^{\rm mut}
      & = & \prod_{k\in E}\frac{\prod_{\ell\in E}(N_{\ell k})_{n_{k\ell}}}{(N_k)_{n_k}}
      \ \sim\ \prod_{k,\ell\in E} (m_{k\ell}(N))^{n_{k\ell}}\\
      & \to & \prod_{k,\ell\in E}(c\rho_{k\ell})^{n_{k\ell}}
      \ =\ a_{\pi\pi'}^{\rm mut}
   \end{eqnarray*}
   as $N\to\infty$. This shows that $P^{\rm mut}\to A^{\rm mut}$ as $N\to\infty$. Similarly, if $\pi\ne\pi'$ are such that each $k$-block of $\pi'$ is a union of some $k$-blocks of $\pi$, then, by (\ref{p_rep}),
   \begin{eqnarray*}
      p_{\pi\pi'}^{\rm rep}
      & = & \prod_{k\in E} \Phi_{j_k}^{(k,N_k)}(i_{k,1},\ldots,i_{k,j_k})\\
      & \to & \prod_{k\in E} \big(cd_k\phi_{j_k}^{(k)}(i_{k,1},\ldots,i_{k,j_k})\big)
      \ =\ a_{\pi\pi'}^{\rm rep}
   \end{eqnarray*}
   as $N\to\infty$, showing that $P^{\rm rep}\to A^{\rm rep}$ as $N\to\infty$. It follows that $P=P^{\rm mut}P^{\rm rep}\to A^{\rm mut}A^{\rm rep}=A$ as $N\to\infty$. Due to the Markov property of the involved processes, the convergence of the finite-dimensional distributions follows immediately. For processes with discrete time set $\nz_0$, the convergence of the finite-dimensional distributions is equivalent (see, for example, Billingsley \cite[p.~19]{Billingsley1999}) to the convergence in $D_{{\cal P}_{n,E}}(\nz_0)$.
\end{proof}
\begin{remark}[Countable type space] \label{infinity} \rm
   Assume that the type space $E$ is countable infinite. Then we conjecture that
   Theorem \ref{main} remains valid under the additional assumption that
   $\rho_k:=\sum_{\ell\in E,\ell\ne k}\rho_{k\ell}<\infty$ for all $k\in E$.
   Since the state space ${\cal P}_{n,E}$ is not finite anymore, one can however not simply follow the proof for the finite type space case. Instead one may verify the relative compactness of the processes $(\Pi_{\lfloor t/c_N\rfloor}^{(n,(N_k)_{k\in E})})_{t\ge 0}$, $(N_k)_{k\in E}$, via similar techniques as in the proof of Theorem 2.1 of Herbots \cite{Herbots1994}. The convergence in $D_{{\cal P}_{n,E}}([0,\infty))$ then follows from Ethier and Kurtz \cite[p.~131, Theorem 7.8]{EthierKurtz1986}.
\end{remark}

\subsection{Multi-type coalescent with mutation} \label{multicoal}
It is not hard to check that the limiting process $\Pi^{(n)}=(\Pi_t^{(n)})_{t\ge 0}$ arising in Theorem \ref{main} is exchangeable in the sense that the distribution of $\Pi^{(n)}$ is invariant under relabelling of the $n$ individuals. We call $\Pi^{(n)}$ a \emph{continuous-time multi-type exchangeable $n$-coalescent with mutation}.

From the consistency relation it follows that the family of processes $\{\Pi^{(n)}:n\in\nz\}$ is consistent, that is, for all $n,m\in\nz$ with $m\le n$, the projected process $(\varrho_{nm}\circ\Pi_t^{(n)})_{t\ge 0}$ has the same distribution as $(\Pi_t^{(m)})_{t\ge 0}$, where $\varrho_{nm}:{\cal P}_{n,E}\to{\cal P}_{m,E}$ denotes the natural projection from ${\cal P}_{n,E}$ to ${\cal P}_{m,E}$ defined via $\varrho_{nm}(\pi):=\{(B_i\cap[m],k_i):1\le i\le j,B_i\cap[m]\ne\emptyset\}$ for all $\pi=\{(B_1,k_1),\ldots,(B_j,k_j)\}\in{\cal P}_{n,E}$. Exploiting Kolmogorov's extension theorem it can be shown that there exists a process $\Pi=(\Pi_t)_{t\ge 0}$ (the projective limit of the sequence $(\Pi^{(n)})_{n\in\nz}$) with state space ${\cal P}_{\infty,E}$, the space of labelled partitions of $\nz$, such that for every $n\in\nz$, $(\varrho_n\circ\Pi_t)_{t\ge 0}$ is a multi-type exchangeable $n$-coalescent with mutation and the same infinitesimal rates, where $\varrho_n:{\cal P}_{\infty,E}\to{\cal P}_{n,E}$ denotes the natural projection from ${\cal P}_{\infty,E}$ to ${\cal P}_{n,E}$ defined via $\varrho_n(\pi):=\{(B\cap[n],k):(B,k)\in\pi,B\cap[n]\ne\emptyset\}$ for all $\pi\in{\cal P}_{\infty,E}$. We call the process $\Pi$ a \emph{multi-type exchangeable coalescent with mutation}.

From the work of Schweinsberg \cite{Schweinsberg2000} it follows that for every $k\in E$ there exists a unique finite measure $\Xi_k$ on the infinite simplex $\Delta:=\{x=(x_i)_{i\in\nz}:x_1\ge x_2\ge\cdots\ge 0,\sum_{i\in\nz}x_i\le 1\}$ such that the infinitesimal rate limits in (\ref{repassumption}) have the integral representation
\begin{eqnarray*}
   &   & \hspace{-15mm}\phi_j^{(k)}(i_1,\ldots,i_j)\ =\ a_k1_{\{j=1,i_1=2\}}\\
   &   & +\int_{\Delta\setminus\{0\}} \sum_{m_1\ne\cdots\ne m_j}x_{m_1}^{i_1}\cdots x_{m_j}^{i_j}\frac{\Xi_k({\rm d}x)}{(x,x)},
\end{eqnarray*}
$i_1,\ldots,i_j\ge 2$, where $a_k:=\Xi_k(\{0\})$ denotes the mass of $\Xi_k$ at $0\in\Delta$ and $(x,x):=\sum_{i\in\nz}x_i^2$ for $x=(x_i)_{i\in\nz}\in\Delta$. The distribution of $\Pi$ is hence fully described by the sequence of measures $\Xi=(\Xi_k)_{k\in E}$, the calibration constants $d_k\ge 0$, $k\in E$, and the mutation parameters $\rho_{k\ell}$, $k\ne\ell$. We therefore call the process $\Pi$ also a (multi-type) $\Xi$-coalescent (with mutation). The process $\Pi$ is the natural generalization of (single-type) exchangeable (and consistent) coalescents to the multi-type case. If all the measures $\Xi_k$, $k\in E$, are concentrated on the subset $[0,1]\times\{0\}\times\{0\}\times\cdots$ of $\Delta$, then we speak of a \emph{multi-type $\Lambda$-coalescent with mutation}, where $\Lambda:=(\Lambda_k)_{k\in E}$ and the finite measure $\Lambda_k$ on $[0,1]$ is defined via $\Lambda_k(B):=\Xi_k(B\times\{0\}\times\{0\}\times\cdots)$ for all Borel sets $B\subseteq[0,1]$.

For a recent paper dealing with multi-type $\Lambda$-coalescents of this form (and as well non-consistent multi-type $\Lambda$-coalescents) we refer the reader to Johnston, Kyprianou and Rogers \cite{JohnstonKyprianouRogers2022}. If $d_k\Xi_k=0$ is the zero measure, then there is no reproductive activity in subpopulation $k$, which may be interpreted as a sleeping seed bank. The ancestral structure of seed bank models has gained some interest in the literature (see, for example, Blath et al. \cite{BlathGonzalezCasanovaKurtSpano2013,BlathGonzalezCasanovaKurtWilkeBerenguer2016} or Gonz\'alez Casanova et al. \cite{GonzalezCasanovaPenalozaSiriJegousse2022}). At this point we also would like to refer the reader to the work of Griffiths \cite{Griffiths2016}, where the notion of a `multi-type $\Lambda$-coalescent' seems to appear for the first time. In different mathematical context, the notion of a `multi-type coalescent point process' also appears in the work of Popovic and Rivas \cite{PopovicRivas2014}, which is mentioned for completeness here. We would also like to draw the attention of the reader to the recent preprint of Allen and McAvoy \cite{AllenMcAvoy2022} for a coalescent with a general spatial and genetic structure and to the work of Liu and Zhou \cite{LiuZhou2022} for a forward stepping stone model with $\Xi$-resampling mechanism.

All what have been said in this subsection applies also to the discrete-time limiting processes $(\Pi_r^{(n)})_{r\in\nz_0}$, $n\in\nz$, arising in Theorem \ref{main2}. Thus there exists a process $(\Pi_r)_{r\in\nz_0}$ with state-space ${\cal P}_{\infty,E}$ such that for every $n\in\nz$ the projected process $(\varrho_n\circ\Pi_r)_{r\in\nz_0}$ is a discrete-time multi-type exchangeable $n$-coalescent with mutation and the same infinitesimal rates. Again, this process can be fully described by a sequence $\Xi=(\Xi_k)_{k\in E}$ of finite measures $\Xi_k$ on the infinite simplex $\Delta$, the calibration constants $d_k\ge 0$, $k\in E$, and the mutation parameters $\rho_{k\ell}\ge 0$, $k\ne\ell$.

\subsection{Multi-type block counting process} \label{blockcounting}

For $t\ge 0$ and $k\in E$ let $N_t^{(k)}$ denote the number of $k$-blocks of $\Pi_t$. Define $N_t:=(N_t^{(k)})_{k\in E}$. Due to a classical criterion of Burke and Rosenblatt \cite[Theorem 1]{BurkeRosenblatt1958}, adapted to the continuous-time setting, the process $(N_t)_{t\ge 0}$, called the \emph{block counting process} of $\Pi$, is Markovian with state space $\nz^E$ and generator $G=G^{\rm rep}+G^{\rm mut}$, where $G^{\rm rep}=(g_{ij}^{\rm rep})_{i,j\in\nz^E}$ and $G^{\rm mut}=(g_{ij}^{\rm mut})_{i,j\in\nz^E}$ have the following entries. Let $i=(i_k)_{k\in E},j=(j_k)_{k\in E}\in\nz^E$ with $i\ge j$. If there exists exactly one $k\in E$ with $i_k>j_k$, then
\begin{eqnarray*}
   g_{ij}^{\rm rep}
   & = & d_k\frac{i_k!}{j_k!}\sum_{i_{k,1},\ldots,i_{k,j_k}}
         \frac{\phi_{j_k}^{(k)}(i_{k,1},\ldots,i_{k,j_k})}
         {i_{k,1}!\cdots i_{k,j_k}!},
\end{eqnarray*}
where the sum extends over all $i_{k,1},\ldots,i_{k,j_k}\in\nz$ satisfying
$i_{k,1}+\cdots+i_{k,j_k}=i_k$. All other non-diagonal entries of $G^{\rm rep}$ are equal to $0$. The diagonal entries are given by
\[
g_{ii}^{\rm rep}\ =\ \sum_{k\in E}d_k\phi_{i_k}^{(k)}(1,\ldots,1),
\]
where $\phi_{i_k}^{(k)}(1,\ldots,1)$ ($\le 0$) is defined via (\ref{phijk}) for all $k\in E$ and $i_k\in\nz$. The matrix $G^{\rm mut}$ has entries
$g_{ij}^{\rm mut}=i_k\rho_{k\ell}$, if $j=i-e_k+e_\ell$ for some $k,\ell\in E$ with $k\ne\ell$, where $e_k$ denotes the $k$-th unit vector in $\rz^E$. For example, if each $\Xi_k$ is the Dirac measure at $0\in\Delta$, then the generator $G$ has entries
\[
g_{ij}\ =\
\left\{
   \begin{array}{cl}
      -\sum_{k\in E}\big(i_k\rho_k+d_k\binom{i_k}{2}\big) & \mbox{if $j=i$,}\\
      \displaystyle i_k\rho_{k\ell} & \mbox{if $j=i-e_k+e_\ell$}\\
      & \hspace{3mm}\mbox{for some $k\ne\ell$,}\\
      d_k\binom{i_k}{2} & \mbox{if $j=i-e_k$}\\
      & \hspace{3mm}\mbox{for some $k\in E$,}\\
      0 & \mbox{otherwise,}
   \end{array}
\right.
\]
$i=(i_k)_{k\in E},j=(j_k)_{k\in E}\in\nz_0^E$. In this case $(N_t)_{t\ge 0}$ coincides with the structured coalescent studied by Notohara \cite{Notohara1990} and Wilkinson--Herbots \cite[Eq.~(3)]{WilkinsonHerbots1998}, where $\rho_{k\ell}\ge 0$ is the mutation rate from type $k$ to type $\ell$ (backward in time), $\rho_k:=\sum_{\ell\ne k}\rho_{k\ell}$ and $d_k\ge 0$ is the coalescence rate of any pair of lineages in subpopulation $k\in E$. In the notation of Johnston, Kyprianou and Rogers \cite[Example 2.2]{JohnstonKyprianouRogers2022}, this structured coalescent is the block counting process of a multi-type Kingman coalescent with binary merging rate $\rho_{kk\to k}:=d_k$ and type changing rate $\rho_{k\to\ell}:=\rho_{k\ell}$.

\section{A multi-type Cannings population model with variable subpopulation sizes} \label{second}

We now study a different but closely related multi-type population model with constant total population size. We consider a population with a fixed number $N\in\nz$ of individuals in each generation $r\in\gz$. As for the model described in the introduction, each individual $i\in[N]$ alive in generation $r\in\gz$ produces a random number $\nu_{i,N}^{(r)}$ of offspring and each offspring a-priori inherits the type of its parent.

In each generation, independently of the offspring sizes, a mutation step follows the reproduction step. Each offspring of type $k\in E$ mutates to type $\ell\in E$ with a given probability $u_{k\ell}\ge 0$. These offspring form the next generation, so our model has non-overlapping generations.

As for the standard (single-type) Cannings model, it is assumed that the offspring sizes are exchangeable within each generation and independent and identically distributed (iid) over different generations. Since the population size is assumed to be constant equal to $N$, the relation $\sum_{i\in[N]}\nu_{i,N}^{(r)}=N$ holds for each generation $r\in\gz$.

As before, the type space $E$ is assumed to be finite or countable infinite. Without loss of generality $E=\{1,\ldots,K\}$ for some $K\in\nz$ or $E=\nz$.
% In this case we also write $u_{k\ell}=Q_k(\{\ell\})$ for the mutation probabilities (per individual per generation) to mutate from type $k\in E$ to type $\ell\in E$.
Clearly, $U:=(u_{k\ell})_{k,\ell\in E}$ is a stochastic matrix, called the mutation matrix. The model is neutral with no selection, since the number of offspring produced by each individual does not dependent of the type of this individual. We call this model the neutral multi-type Cannings model with mutation. For $|E|=1$ the model reduces to the classical neutral exchangeable population model of Cannings \cite{Cannings1974,Cannings1975,Cannings1976} as described in the introduction. We use the notation $\nu_{i,N}:=\nu_{i,N}^{(0)}$, $C_0:=0$ and $C_i:=\nu_{1,N}+\cdots+\nu_{i,N}$ for $i\in[N]$.

\subsection{Forward structure} \label{forwardstructure}
Let $X_k(r)$ denote the number of individuals of type $k\in E$ in generation $r\in\nz_0$, and set $X(r):=(X_k(r))_{k\in E}$. It is readily seen that $X:=(X(r))_{r\in\nz_0}$ is a time-homogeneous Markov chain with state space $\Delta_N(E):=\{i=(i_k)_{k\in E}\in\nz_0^E:\sum_{k\in E}i_k=N\}$. Note that $|\Delta_N(E)|=\binom{N+|E|-1}{N}$ if $|E|<\infty$.

Let $\Pi:=(\pi_{ij})_{i,j\in\Delta_N(E)}$ denote the transition matrix of $X$ having entries $\pi_{ij}:=\pr(X(r+1)=j\,|\,X(r)=i)$. Clearly $\Pi=\Pi(U)$ depends on the mutation matrix $U=(u_{k\ell})_{k,\ell\in E}$. For the multi-allelic Cannings model without mutation (when $U=I$ is the identity matrix) it is known (see, for example, Gladstien \cite{Gladstien1978} or \cite[Eq.~(1)]{Moehle2010}) that $\Pi^{\rm rep}:=\Pi(I)$ has entries
\begin{equation} \label{transnomut}
   \pi_{ij}^{\rm rep}\ =\ \pi_{ij}(I)\ =\ \pr(D(i)=j),\qquad i,j\in\Delta_N(E),
\end{equation}
where $D(i):=(D_k(i))_{k\in E}$ is defined via
\[
D_k(i)\ :=\ \sum_{s=s_{k-1}+1}^{s_k}\nu_{s,N},
\qquad k\in E,
\]
with $s_0:=0$ and $s_k:=i_1+\cdots+i_k$ for $k\in E$. Note that $\sum_{k\in E}D_k(i)=\nu_{1,N}+\cdots+\nu_{N,N}=N$. In particular, $\pi_{ii}(I)=1$ for all states
$i$ of the form $i=Ne_k$ with $k\in E$, where $e_k$ denotes the $k$-th unit vector in $\rz^E$. Thus all the states $Ne_k$, $k\in E$, are absorbing.

In order to describe the structure of the transition matrix $\Pi=\Pi(U)$ for general mutation matrix $U$ it turns out to be convenient to introduce a matrix $\Pi^{\rm mut}$ as follows. Fix $i=(i_k)_{k\in E}\in\Delta_N(E)$ and let $M_k$, $k\in E$, be independent random variables, where $M_k=(M_{k\ell})_{\ell\in E}$ has a multinomial distribution with parameters $i_k$ and $(u_{k\ell})_{\ell\in E}$, i.e. $\pr(M_k=j)=i_k!\prod_{\ell\in E}u_{k\ell}^{j_\ell}/j_\ell!$ for all $j=(j_\ell)_{\ell\in E}\in\nz_0^E$ with $\sum_{\ell\in E} j_\ell=i_k$. Let $\Pi^{\rm mut}=(\pi_{ij}^{\rm mut})_{i,j\in\Delta_N(E)}$ denote the matrix with entries
\begin{equation} \label{aij1}
   \pi_{ij}^{\rm mut}\ :=\ \pr(M_\bullet=j),
   \qquad i,j\in\Delta_N(E),
\end{equation}
where we use the dot subscript summation notation $M_\bullet:=\sum_{k\in E}M_k$. It is readily checked that
\begin{eqnarray}
   \pi_{ij}^{\rm mut}
   & = & \sum_M \prod_{k\in E}\pr(M_k=(m_{k\ell})_{\ell\in E})\nonumber\\
   & = & \sum_M \prod_{k\in E}
   \bigg(i_k!\prod_{\ell\in E}\frac{u_{k\ell}^{m_{k\ell}}}{m_{k\ell}!}\bigg) \label{aij2}
\end{eqnarray}
for all $i,j\in\Delta_N(E)$, where the sum extends over all matrices $M=(m_{k\ell})_{k,\ell\in E}\in\nz_0^{E\times E}$ having row sums $m_{k\bullet}:=\sum_{\ell\in E}m_{k\ell}=i_k$, $k\in E$, and column sums $m_{\bullet\ell}:=\sum_{k\in E}m_{k\ell}=j_\ell$, $\ell\in E$. Clearly, the matrix $\Pi^{\rm mut}$ depends on the mutation matrix $U=(u_{k\ell})_{k,\ell\in E}$. For example, for $E=\{1,2\}$ (two types) it follows that
\begin{eqnarray*}
   \pi_{ij}^{\rm mut}
   & = & \pr(M_{11}+M_{21}=j_1,M_{12}+M_{22}=j_2)\\
   & = & \pr(M_{11}+M_{21}=j_1),
\end{eqnarray*}
for all $i=(i_1,i_2)\in\nz_0^2$ and $j=(j_1,j_2)\in\nz_0^2$ with $i_1+i_2=N$ and $j_1+j_2=N$, where $M_{11}$ and $M_{21}$ are independent, $M_{11}$ has a binomial distribution with parameters $i_1$ and $u_{11}$ and $M_{21}$ has a binomial distribution with parameters $i_2=N-i_1$ and $u_{21}$.

The following result clarifies the structure of the transition matrix $\Pi=\Pi(U)$ for general mutation matrix $U$.
\begin{lemma} \label{lemma1}
   The chain $X$ (of the model with general mutation matrix $U$) has transition matrix
   \begin{equation} \label{forwardmatrix}
      \Pi\ =\ \Pi^{\rm rep}\Pi^{\rm mut},
   \end{equation}
   where $\Pi^{\rm rep}:=\Pi(I)$ is the transition matrix of the chain $X$ for the multi-type Cannings model without mutation having entries (\ref{transnomut}) and $\Pi^{\rm mut}$ is the matrix with entries (\ref{aij2}).
\end{lemma}
\begin{remark} \rm
   The transition matrix $\Pi$ thus factorizes into the reproductive part $\Pi^{\rm rep}$ (involving the offspring sizes $\nu_{1,N},\ldots,\nu_{N,N}$) and the mutational part $\Pi^{\rm mut}$ (involving the mutation matrix $U$).
\end{remark}
\begin{proof}[Proof of Lemma \ref{lemma1}]
   Let $i=(i_k)_{k\in E},j=(j_k)_{k\in E}\in\Delta_N(E)$. By the independence of $\nu_N^{(r)}:=(\nu_{1,N}^{(r)},\ldots,\nu_{N,N}^{(r)})$ and $X(r)$,
   \[
   \pi_{ij}\ =\ \sum_m \pr(X(r+1)=j\,|\,X(r)=i,\nu_N^{(r)}=m)\pr(\nu_N^{(r)}=m),
   \]
   where the sum extends over all $m=(m_1,\ldots,m_N)\in\nz_0^N$ satisfying $\pr(\nu_N^{(r)}=m)>0$. Let $I_k:=\{s\in[N]:\mbox{individual $s$ in generation $r$ has type $k$}\}$, $k\in E$. Then,
   \begin{eqnarray*}
      &   & \hspace{-15mm}\pr(X(r+1)=j\,|\,X(r)=i,\nu_N^{(r)}=m)\\
      & = & \sum_M \prod_{k\in E} \bigg(\Big(\sum_{s\in I_k}m_s\Big)!\prod_{\ell\in E}\frac{u_{k\ell}^{m_{k\ell}}}{m_{k\ell}!}\bigg),
   \end{eqnarray*}
   where $m_{k\ell}$ corresponds to the number of children of parents of type $k$ which mutate to type $\ell$. Note that $\sum_{k\in E}m_{k\ell}=j_\ell$, $\ell\in E$, and $\sum_{\ell\in E}m_{k\ell}=\sum_{s\in I_k}m_s$, $k\in E$. The offspring variables $\nu_{1,N},\ldots,\nu_{N,N}$ are exchangeable and the sets $I_k$, $k\in E$, are pairwise disjoint with $\bigcup_{k\in E}I_k=[N]$ and $|I_k|=i_k$, $k\in E$. Therefore,
   \begin{equation} \label{transalt}
      \pi_{ij}
      \ = \ \me\bigg(
            \sum_M\prod_{k\in E}
                \bigg(
                   D_k(i)!\prod_{\ell\in E}\frac{u_{kl}^{m_{k\ell}}}{m_{k\ell}!}
               \bigg)
            \bigg),
   \end{equation}
   where the sum extends over all matrices $M=(m_{k\ell})_{k,\ell\in E}\in\nz_0^{E\times E}$ having column sums $\sum_{k\in E}m_{k\ell}=j_\ell$, $\ell\in E$, and row sums $\sum_{\ell\in E}m_{k\ell}=D_k(i)$, $k\in E$. Alternatively,
   \begin{equation} \label{trans}
      \pi_{ij}\ =\ \sum_M \pr(D(i)=d)\,
      \prod_{k\in E}\bigg(
         d_k!\prod_{\ell\in E}\frac{u_{k\ell}^{m_{k\ell}}}{m_{k\ell}!}
      \bigg),
   \end{equation}
   where the sum extends over all matrices $M=(m_{k\ell})_{k,\ell\in E}\in\nz_0^{E\times E}$ having column sums $\sum_{k\in E}m_{k\ell}=j_\ell$, $\ell\in E$, and $d:=(d_k)_{k\in E}$ is defined via $d_k:=\sum_{\ell\in E}m_{k\ell}$, $k\in E$. It remains to note that the right-hand side of Eq.~(\ref{trans}) is equal to $\sum_{d\in\Delta_N(E)}\pi_{id}^{\rm rep}\pi_{dj}^{\rm mut}=(\Pi^{\rm rep}\Pi^{\rm mut})_{ij}$. Thus, $\Pi=\Pi^{\rm rep}\Pi^{\rm mut}$.
\end{proof}

The transition probability $\pi_{ij}$ simplifies considerably in particular situations, as the following examples demonstrate.
\begin{example} (parent independent mutation)
   If the mutation probabilities $u_{k\ell}=u_\ell$ do not depend on the type $k\in E$ of the parent, then (\ref{trans}) reduces to the multinomial expression
   \begin{equation} \label{trans1}
      \pi_{ij}\ =\ N!\prod_{k\in E}\frac{u_k^{j_k}}{j_k!},
   \end{equation}
   $i=(i_k)_{k\in E},j=(j_k)_{k\in E}\in \Delta_N(E)$. Eq.~(\ref{trans1}) is obvious from the model and can be also derived from (\ref{forwardmatrix}) as follows. By the convolution property for multinomial distributions ${\rm Mn}(i_k,u)$, $k\in E$, with the same probability vector $u:=(u_\ell)_{\ell\in E}$, $M_\bullet
   =\sum_{k\in E}M_k$ has a multinomial distribution with parameters $\sum_{k\in E}i_k=N$ and $u$. Thus, by (\ref{aij1}), $\pi_{ij}^{\rm mut}=\pr(M_\bullet=j)={\rm Mn}(N,u)(j)$ does not depend on $i$ and (\ref{trans1}) follows from (\ref{forwardmatrix}). Note that the transition probability $\pi_{ij}$ in (\ref{trans1}) neither depends on $i$ nor on the offspring sizes $\nu_{1,N},\ldots,\nu_{N,N}$. If $K:=|E|\in\nz$ and $u_{k\ell}:=1/K$ for all $k,\ell\in E$ then $\pi_{ij}=N!K^{-N}/(j_1!\cdots j_K!)$.
\end{example}
\begin{example} (absence of mutation)
   If $U=I$ (identity matrix), which corresponds to the multi-allelic Cannings model without mutation, then the matrix $\Pi^{\rm mut}$ with entries (\ref{aij1}) is the identity matrix and the transition matrix $\Pi=\Pi^{\rm rep}\Pi^{\rm mut}=\Pi^{\rm rep}$ has entries (\ref{transnomut}).
\end{example}
\begin{example}
   (multi-allelic Wright--Fisher model with mutation) If the family size vector $\nu_N:=(\nu_{1,N},\ldots,\nu_{N,N})$ has a symmetric multinomial distribution ${\rm Mn}(N,(1/N)_{i\in[N]})$, then $D(i)=(D_k(i))_{k\in E}\stackrel{d}{=}{\rm Mn}(N,(i_k/N)_{k\in E})$. Plugging $\pr(D(i)=d)=N!N^{-N}\prod_{k\in E}i_k^{d_k}/d_k!$ into (\ref{trans}), the factorials $d_k!$ cancel, and from $\sum_{l\in E}m_{kl}=d_k$ it follows that
   \[
   \pi_{ij}\ =\ \frac{N!}{N^N}\sum_M \prod_{k\in E}\prod_{\ell\in E} \frac{(u_{k\ell}i_k)^{m_{k\ell}}}{m_{k\ell}!},
   \]
   where the sum extends over all matrices $M=(m_{k\ell})_{k,\ell\in E}$ with $\sum_{k\in E}m_{k\ell}=j_\ell$, $\ell\in E$. Applying the binomial expansion $n!\sum_{(n_k)_{k\in E}\in\Delta_n(E)}\prod_{k\in E}(x_k^{n_k}/n_k!)=(\sum_{k\in E}x_k)^n$, $n\in\nz_0$, $x_k\in\rz$, $k\in E$, for each $\ell\in E$ with $n:=j_\ell$, $n_k:=m_{k\ell}$ and $x_k:=u_{k\ell}i_k$ shows that (\ref{trans}) simplifies to
   \begin{equation} \label{wfm_trans}
      \pi_{ij}
      \ =\ N!\prod_{\ell\in E}\frac{\pi_\ell^{j_\ell}}{j_\ell!},
      \qquad i,j\in\Delta_N(E),
   \end{equation}
   with $\pi_\ell:=N^{-1}\sum_{k\in E}u_{k\ell}i_k$ for $\ell\in E$. In particular, conditional on $X(r)=i$, for every $\ell\in E$ the random variable $X_\ell(r+1)$ has a binomial distribution with parameters $N$ and $\pi_\ell$, converging as $N\to\infty$ in distribution to a Poisson distribution with parameter $N\pi_\ell=\sum_{k\in E}u_{k\ell}i_k$. For $|E|<\infty$ the chain $X$ and its diffusion limit as $N\to\infty$ of this multi-allele Wright--Fisher model with mutation has been studied extensively in the (classical) literature on mathematical population genetics. We refer the reader exemplary to Etheridge \cite{Etheridge2011}, Ewens \cite{Ewens2004} and Griffiths \cite{Griffiths1979,Griffiths1980a,Griffiths1980b}.
\end{example}
A further example, the multitype Kimura model, is deferred to Subsection \ref{kimuramodel}.

The forward structure is often alternatively described by the process $\widetilde{X}:=(\widetilde{X}(r))_{r\in\nz_0}$ having state space $E^N$, where $\widetilde{X}(r):=(\widetilde{X}_i(r))_{i\in[N]}$ and $\widetilde{X}_i(r)$ denotes the type of the \mbox{$i$-th} individual alive in generation $r\in\nz_0$. Processes of this form are extensively used, for example in Birkner et al. \cite{BirknerBlathMoehleSteinrueckenTams2009} even for more general type spaces $E$. Let $f:E^N\to\Delta_N(E)$ denote the function which maps $x\in E^N$ to $f(x):=i:=(i_k)_{k\in E}$ with $i_k:=|\{i\in[N]:x_i=k\}|$, $k\in E$. The process $X$ is easily pathwise recovered from $\widetilde{X}$ via $X(r)=f(\widetilde{X}(r))$, $r\in\nz_0$. Strictly speaking, the process $\widetilde{X}$ contains (pathwise) slightly more information than $X$, since the type of each individual is known. However, due to the exchangeability of the model arising from the random assignment condition, from the distributional point of view the processes $X$ and $\widetilde{X}$ contain the same information. More precisely, the transition probabilities $\widetilde\pi_{xy}:=\pr(\widetilde{X}(r+1)=y\,|\,\widetilde{X}(r)=x)$, $x,y\in E^N$,   of the process $\widetilde{X}$ are related to those of $X$ via
\begin{equation} \label{transrelation}
   \widetilde\pi_{xy}
   \ =\ \frac{\pi_{ij}}{|f^{-1}(\{j\})|}
   \ =\ \frac{\prod_{k\in E}j_k!}{N!}\pi_{ij},
\end{equation}
where $i:=(i_k)_{k\in E}:=f(x)$ and $j:=(j_k)_{k\in E}:=f(y)$. Alternatively,
\begin{equation} \label{tildetrans}
   \widetilde\pi_{xy}
   \ =\ \frac{1}{N!}\sum_{\sigma\in S_N}\me\bigg(
   \prod_{i=1}^N \prod_{j=C_{i-1}+1}^{C_i}u_{x_iy_{\sigma j}}
   \bigg),
\end{equation}
where $S_N$ denotes the set of all permutations of $[N]$, $C_0:=0$ and $C_i:=\nu_{1,N}+\cdots+\nu_{i,N}$ for all $i\in[N]$. For $N=1$ we recover the mutation probabilities $\widetilde\pi_{xy}=u_{xy}$, $x,y\in E$.

For parent independent mutation, that is, $u_{k\ell}=u_\ell$ for all $k,\ell\in E$, Eq.~(\ref{tildetrans}) reduces to
\[
\widetilde\pi_{xy}
\ =\ \frac{1}{N!}\sum_{\sigma\in S_N}\me\bigg(\prod_{j=1}^N u_{y_{\sigma j}}\bigg)
\ =\ u_{y_1}\cdots u_{y_N},
\]
which neither depends on $i$ nor on $(\nu_{1,N},\ldots,\nu_{N,N})$.

In absence of mutation ($u_{kk}=1$ for all $k\in E$),
\[
\widetilde\pi_{xy}\ =\ \frac{\prod_{k\in E}j_k!}{N!}\pr(D(i)=j),
\qquad x,y\in E^N,
\]
where $i:=(i_k)_{k\in E}:=f(x)$ and $j:=(j_k)_{k\in E}:=f(y)$ and $D(i):=(D_k(i))_{k\in E}$ with $D_k(i):=\sum_{s=s_{k-1}+1}^{s_k}\nu_{s,N}$, $s_0:=0$ and $s_k:=i_1+\cdots+i_k$ for $k\in E$.

For the multi-type Wright--Fisher model with mutation, each child $j\in[N]$ chooses randomly and independently its parent. Thus,
\begin{eqnarray}
   \widetilde\pi_{xy}
   & = & \prod_{j\in[N]} \sum_{i\in[N]} \pr(\mbox{child $j$ chooses parent $i$})u_{x_iy_j}\nonumber\\
   & = & \frac{1}{N^N}\prod_{j\in[N]} \sum_{i\in[N]} u_{x_iy_j}
\qquad x,y\in E^N.
\end{eqnarray}
Fix $x\in E^N$ and $j\in\Delta_N(E)$. Define $i:=f(x)$. Then,
\begin{eqnarray}
   &   & \hspace{-15mm}\sum_{y\in f^{-1}(\{j\})} \widetilde\pi_{xy}\nonumber\\
   & = & \frac{1}{N^N}\sum_{y\in f^{-1}(\{j\})}
         \prod_{n\in[N]}
         (u_{x_1y_n}+\cdots+u_{x_Ny_n})
         \nonumber\\
   & = & \frac{1}{N^N}\sum_{y\in f^{-1}(\{j\})}
         \prod_{n\in[N]}\bigg(\sum_{k\in E}i_ku_{ky_n}\bigg)
         \nonumber\\
   & = & \frac{1}{N^N}\sum_{y\in f^{-1}(\{j\})}
         \prod_{\ell\in E} \bigg(\sum_{k\in E}i_ku_{k\ell}\bigg)^{j_\ell}
         \nonumber\\
   & = & \sum_{y\in f^{-1}(\{j\})}
         \prod_{\ell\in E}\pi_\ell^{j_\ell}\nonumber\\
   & = & |f^{-1}(\{j\})|\prod_{\ell\in E}\pi_\ell^{j_\ell}
   \ = \ N!\prod_{\ell\in E}\frac{\pi_\ell^{j_\ell}}{j_\ell!}.
   \label{rosenblatt}
\end{eqnarray}
This expression depends only via $i=f(x)$ on $x$. Thus, by Rosenblatt's criterion for functions of Markov processes, with $\tilde{X}$ also the process $X=(X(r))_{r\in\nz_0}=(f(\widetilde{X}(r)))_{r\in\nz_0}$ is Markovian and (\ref{rosenblatt}) is the transition probability $\pi_{ij}$ of the chain $X$ to move from $i$ to $j$, in agreement with (\ref{wfm_trans}).

%Define the empirical random measure
%\begin{equation}
%   Z_r^{(N)}:=\frac{1}{N}\sum_{i=1}^N \delta_{\widetilde{X}_i^{(N)}(r)}.
%\end{equation}
%Note that $Z_r^{(N)}$ is a random probability measure on $E$, so $Z_r^{(N)}:\Omega\to{\cal M}_1(E)$, where ${\cal M}_1(E)$ denotes the space of all probability measures on $E$.
%Forward result:
%\begin{theorem}
%   Assume that $c_N\to0$ as $N\to\infty$.
%   If $Z_0^{(N)}$ converges in distribution to some
%   $Z_0$ as $N\to\infty$.
%   Then, as $N\to\infty$, the time-scaled process
%   $(Z_{\lfloor t/c_N\rfloor}^{(N)})_{t\ge 0}$ converges
%   in $D_{{\cal M}_1(E)}([0,\infty))$ to the
%   $(\Xi,B)$-Fleming-Viot process $(Z_t)_{t\ge 0}$.
%\end{theorem}

\subsection{A limiting multi-type branching process} \label{branching}
Assume that the number of type is finite but not equal to $1$; without loss of generality, $E=[K]$ with $K\in\nz\setminus\{1\}$. Define $L:=K-1\in\nz$. In the following the space $S_{N,L}:=S_N([L])=\{i=(i_\ell)_{\ell\in[L]}\in\nz_0^L\,:\,\sum_{\ell\in[L]}i_\ell\le N\}$ plays a crucial role. So far, the model was described forward in time by the process $X=(X(r))_{r\in\nz_0}$ having state space $\Delta_{N,K}:=\Delta_N([K])=\{i=(i_k)_{k\in E}\in\nz_0^K:\sum_{k\in[K]}i_k=N\}$. Note that $|S_{N,L}|=\sum_{n=0}^N|\Delta_{n,L}|=\sum_{n=0}^N \binom{n+L-1}{n}=\binom{N+L}{N}=\binom{N+K-1}{N}=|\Delta_{N,K}|$. Since the last component $X_K(r)=N-\sum_{\ell=1}^LX_\ell(r)$ is determined by the other components $X_1(r),\ldots,X_L(r)$, we can disregard the last component $X_K(r)$ and view the forward process $X=(X(r))_{r\in\nz_0}$ as a process with state space $S_{N,L}$ by writing $X(r)$ equivalently in the form
\begin{equation} \label{forward}
   X_r\ :=\ (X_\ell(r))_{\ell\in[L]},\qquad r\in\nz_0.
\end{equation}
The following definition is useful to understand the behavior of $X_r$ as the total population $N$ tends to infinity.
\begin{definition} % [Asymptotic independence]
%   \ \\
   The offspring sizes $\nu_{1,N},\ldots,\nu_{N,N}$ are said to be asymptotically independent, if there exist independent random variables $\xi_1,\xi_2,\ldots$ such that
   \begin{equation}
      \lim_{N\to\infty}\pr(\nu_{1,N}=m_1,\ldots,\nu_{j,N}=m_j)
      \ =\ \prod_{i=1}^j\pr(\xi_i=m_i)
   \end{equation}
   for all $j\in\nz$ and $m_1,\ldots,m_j\in\nz_0$.
\end{definition}
The exchangeability of $\nu_{1,N},\ldots,\nu_{N,N}$ implies that the random variables $\xi_1,\xi_2,\ldots$ are identically distributed. As in \cite[p.~491]{Kaemmerle1991} or \cite[Lemma 2.1]{Moehle1994} it follows that $\nu_{1,N},\ldots,\nu_{N,N}$ are asymptotically independent if and only if there exists a random variable $\xi$ such that, for all $m\in\nz_0$, $\pr(\nu_{1,N}=m)\to\pr(\xi=m)$ and $\pr(\nu_{1,N}=\nu_{2,N}=m)\to(\pr(\xi=m))^2$ as $N\to\infty$. We call $\xi$ the limiting variable. Note that $\me(\xi)\le 1$.
There exist examples with $\me(\xi)<1$. For example, if $(\nu_{1,N},\ldots,\nu_{N,N})$ is a random permutation of $(N,0,\ldots,0)$, then $\nu_{1,N},\ldots,\nu_{N,N}$ are asymptotically independent with limiting variable $\xi=0$.

We shall see soon that the asymptotic independence of the offspring sizes alone will not be sufficient to ensure convergence of the forward process to a limiting process as $N\to\infty$. An additional assumption on the mutation probabilities is needed, which turns out to be of the form
\begin{equation} \label{additional}
   u_{KK}\ =\ 1,
\end{equation}
that is, mutations from the `last' type $K$ back to any other type $k<K$ are impossible. Clearly, (\ref{additional}) wipes out several important models. The remark at the end of this section shows that the following results fail if (\ref{additional}) does not hold. The convergence results presented later (Theorem \ref{gwplimit}) are based on the following lemma.
\begin{lemma} \label{gwplemma}
   Suppose that the offspring sizes $\nu_{1,N},\ldots,\nu_{N,N}$ are asymptotically independent with limiting variable $\xi$. Let $K\in\nz\setminus\{1\}$ and set $L:=K-1\in\nz$. If $u_{KK}=1$, that is, mutations from type $K$ to any type $\ell\in[L]$ are not possible, then
   \begin{equation} \label{gwpassumption}
      \pr(X_1=j\,|\,X_0=i)
      \ \to \ \pr\bigg(\bigcap_{\ell\in[L]}\bigg\{\sum_{k\in[L]}
           Y_{k\ell}^{*i_k}=j_\ell\bigg\}\bigg)
   \end{equation}
   as $N\to\infty$ for all $i=(i_\ell)_{\ell\in[L]},j=(j_\ell)_{\ell\in[L]}\in\nz_0^L$, where $Y_{k\ell}^{*i_k}$ denotes the $i_k$-th convolution of $Y_{k\ell}$ and $Y_k=(Y_{k\ell})_{\ell\in[L]}$, $k\in[L]$, are independent $\nz_0^L$-valued random variables having distribution
\begin{equation} \label{ykdist}
   \pr(Y_k=j)
   \ =\ \me\big((\xi)_{|j|}u_{kK}^{\xi-|j|}\big)
   \prod_{\ell\in[L]}\frac{u_{k\ell}^{j_\ell}}{j_\ell!},
\end{equation}
$j=(j_\ell)_{\ell\in[L]}\in\nz_0^L$,
with $|j|:=\sum_{\ell\in[L]}j_\ell$. In (\ref{ykdist}), the notation $(x)_0:=1$ and $(x)_n:=x(x-1)\cdots(x-n+1)$ for $x\in\rz$ and $n\in\nz$ is used.
\end{lemma}
\begin{remark} \rm
   Fix $k\in[L]$. Eq.~(\ref{ykdist}) essentially states that, conditional on $\xi$, $(Y_{k1},\ldots,Y_{kL},\xi-\sum_{\ell\in[L]} Y_{k\ell})$ has a multinomial distribution with parameters $\xi$ and $(u_{k\ell})_{\ell\in E}$. By the convolution property for multinomial distributions with the same probability vector, conditional on $\xi$, for any $n\in\nz_0$, the $n$-th convolution $(Y_{k1}^{*n},\ldots,Y_{kL}^{*n},(\xi-\sum_{\ell\in [L]}Y_{k\ell})^{*n})$ has a multinomial distribution with parameters $\xi^{*n}$ and $(u_{k\ell})_{\ell\in E}$. In particular, for all $n\in\nz_0$ and $(j_\ell)_{\ell\in[L]}\in\nz_0^L$,
   \[
   \pr\bigg(\bigcap_{\ell\in[L]}\{Y_{k\ell}^{*n}=j_\ell\}\bigg)\ =\ \me((\xi^{*n})_{|j|}u_{kK}^{\xi^{*n}-|j|})\prod_{\ell\in[L]}
   \frac{u_{k\ell}^{j_\ell}}{j_\ell!},
   \]
   where $|j|:=\sum_{\ell\in[L]} j_\ell$. This convolution property turns out to be crucial for the proof of Lemma \ref{gwplemma} and ensures the branching property of the limiting process arising in Theorem \ref{gwplimit} below.
\end{remark}
\begin{remark} \rm
   For $K=2$, (\ref{gwpassumption}) reduces to $\pr(X_1(1)=j\,|\,X_1(0)=i)\to\pr(Y^{*i}=j)$ as $N\to\infty$ for all $i,j\in\nz_0$, where $Y:=Y_{11}$ has distribution $\pr(Y=j)=\me\big(\binom{\xi}{j}u_{11}^ju_{12}^{\xi-j}\big)$, $j\in\nz_0$. Conditional on $\xi$, $Y$ has a binomial distribution with parameters $\xi$ and $u_{11}$.
\end{remark}
\begin{proof}[Proof of Lemma \ref{gwplemma}]
   Fix $(i_\ell)_{\ell\in[L]}, (j_\ell)_{\ell\in[L]}\in\nz_0^L$ and let $N$ be sufficiently large such that $i_K:=N-\sum_{\ell\in[L]}i_\ell\in\nz_0$ and $j_K:=N-\sum_{\ell\in[L]}j_\ell\in\nz_0$. Define $i:=(i_k)_{k\in E}$ and $j:=(j_k)_{k\in E}$. We have to verify that $\pi_{ij}$ converges to the right-hand side of (\ref{gwpassumption}) as $N\to\infty$. Since $u_{KK}=1$ and, hence, $u_{K\ell}=0$ for all $\ell\in[L]$, it follows that only matrices $M=(m_{k\ell})_{k,\ell\in E}$ with last row $(m_{K\ell})_{\ell\in[K]}$ equal to $(0,\ldots,0,D_k)$ contribute to the sum in (\ref{transalt}). Thus, (\ref{transalt}) reduces to
   \[
   \pi_{ij}\ =\
   \me\bigg(
      \sum_M \prod_{k\in[L]}\bigg(
      (D_k)_{m_k}u_{kK}^{D_k-m_k}
      \prod_{\ell\in[L]}\frac{u_{k\ell}^{m_{k\ell}}}{m_{k\ell}!}
      \bigg)
   \bigg),
   \]
   where $m_k:=\sum_{\ell\in[L]}m_{k\ell}$ and the sum extends over all `reduced' matrices $M=(m_{k\ell})_{k,\ell\in[L]}\in\nz_0^{L\times L}$ having column sums $\sum_{k\in E}m_{k\ell}=j_\ell$, $\ell\in[L]$. From the asymptotic independence of $\nu_{1,N},\ldots,\nu_{N,N}$ it follows that $D=(D_k)_{k\in[L]}\to (\xi_k^{*i_k})_{k\in[L]}$ in distribution as $N\to\infty$, where $\xi_1,\xi_2,\ldots$ are iid copies of a random variable $\xi$ satisfying $\nu_1\to\xi$ in distribution as $N\to\infty$. Thus, as $N\to\infty$, $\pi_{ij}$ converges to
   \begin{eqnarray*}
      &   & \hspace{-20mm}\sum_M \prod_{k\in[L]} \bigg(\me\big((\xi_k^{*i_k})_{m_k}u_{kK}^{\xi_k^{*i_k}-m_k}\big)
      \prod_{\ell\in[L]}\frac{u_{k\ell}^{m_{k\ell}}}{m_{k\ell}!}\bigg)\\
      & = & \sum_M \prod_{k\in[L]}\pr\bigg(\bigcap_{l\in[L]}\{Y_{k\ell}^{*i_k}=m_{k\ell}\}\bigg)\\
      & = & \sum_M \pr\bigg(\bigcap_{k,\ell\in[L]}\{Y_{k\ell}^{*i_k}=m_{k\ell}\}\bigg)\\
      & = & \pr\bigg(
            \bigcap_{\ell\in[L]}\bigg\{\sum_{k\in[L]} Y_{k\ell}^{*i_k}=j_\ell\bigg\}
            \bigg).
   \end{eqnarray*}
   The proof is complete.
\end{proof}
Before we come to the main convergence result, we provide some more information on $Y_k$.
\begin{lemma}[Descending factorial moments of $Y_k$]
   \label{facmom}
   The joint descending factorial moments $\mu(m):=\me(\prod_{\ell\in[L]}(Y_{k\ell})_{m_\ell})$, $m=(m_\ell)_{\ell\in[L]}\in\nz_0^L$, of $Y_k=(Y_{k\ell})_{\ell\in[L]}$ are given by
   \begin{equation}
      \mu(m)
      \ =\ \me\big((\xi)_{|m|}\big)\prod_{\ell\in[L]}u_{k\ell}^{m_\ell},
   \end{equation}
   where $|m|:=\sum_{\ell\in[L]} m_\ell$.
\end{lemma}
\begin{proof}[Proof of Lemma \ref{facmom}]
   Let $k\in[L]$ and $m=(m_\ell)_{\ell\in[L]}\in\nz_0^L$. By (\ref{ykdist}),
   \begin{eqnarray*}
      &   & \hspace{-10mm}\mu(m)
      \ = \ \sum_{j_1,\ldots,j_L\in\nz_0}
            \pr(Y_k=(j_1,\ldots,j_L))
            \prod_{\ell\in[L]}(j_\ell)_{m_\ell}\\
      & = & \sum_{j\in\nz_0}\me\big((\xi)_ju_{kK}^{\xi-j}\big)
            \sum_{\substack{j_1,\ldots,j_L\in\nz_0\\j_1+\cdots+j_L=j}}
            \prod_{\ell\in[L]}\frac{(j_\ell)_{m_\ell}u_{k\ell}^{j_\ell}}{j_\ell!}\\
      & = & \sum_{j\ge|m|}
            \me\big((\xi)_ju_{kK}^{\xi-j}\big)
            \sum_{\substack{j_1\ge m_1,\ldots,j_L\ge m_L\\j_1+\cdots+j_L=j}}
            \prod_{\ell\in[L]}\frac{u_{k\ell}^{j_\ell}}{(j_\ell-m_\ell)!}.
   \end{eqnarray*}
   By the multinomial formula, the last sum reduces to $u_k^{j-|m|}/(j-|m|)!\prod_{\ell\in[L]}u_{k\ell}^{j_\ell}$, where $u_k:=\sum_{\ell\in[L]}u_{k\ell}$. Thus,
   \begin{eqnarray*}
      \mu(m)
      & = & \sum_{j\ge|m|}\me\big((\xi)_ju_{kK}^{\xi-j}\big)
            \frac{u_k^{j-|m|}}{(j-|m|)!}
            \prod_{\ell\in[L]}u_{k\ell}^{m_\ell}\\
      & = & \me\bigg((\xi)_{|m|}
               \sum_{j\ge|m|} \binom{\xi-|m|}{j-|m|} u_{kK}^{\xi-j}
               u_k^{j-|m|}
            \bigg)\prod_{\ell\in[L]}u_{k\ell}^{m_\ell}.
   \end{eqnarray*}
   The result follows from $\sum_{j\ge|m|}\binom{\xi-|m|}{j-|m|}u_{kK}^{\xi-j}u_k^{j-|m|}=(u_{kK}+u_k)^{\xi-|m|}=1^{\xi-|m|}=1$.
\end{proof}
\begin{remark} \rm
   Note that $Y_{k\ell}$ has descending factorial moments $\me((Y_{k\ell})_m)=\me((\xi)_m)u_{k\ell}^m$, $m\in\nz_0$, in agreement with the fact that, conditional on $\xi$, $Y_{k\ell}$ has a binomial distribution with parameters $\xi$ and $u_{k\ell}$. Moreover, for $\ell_1,\ell_2\in[L]$ with $\ell_1\ne\ell_2$,
   \begin{equation} \label{cov}
      {\rm Cov}(Y_{k\ell_1},Y_{k\ell_2})
      \ =\ \big(\me((\xi)_2)-(\me(\xi))^2\big)u_{k\ell_1}u_{k\ell_2}.
   \end{equation}
   Thus, in general, the random variables $Y_{k\ell}$, $\ell\in[L]$, are not pairwise uncorrelated. A nice exception occurs when $\xi$ is Poisson distributed, since $\me((\xi)_2)=(\me(\xi))^2$ in this case. A concrete such exception is the Wright--Fisher model.
\end{remark}
The following result (Theorem \ref{gwplimit}) shows that, under the assumptions of Lemma \ref{gwplemma}, the forward structure can be approximated for large $N$ by a multi-type Galton--Watson branching process. For general information on multi-type branching processes we refer the reader to Athreya and Ney \cite[Chapter V]{AthreyaNey1972}, Harris \cite[Chapter II]{Harris1963} and Mode \cite{Mode1971}.
\begin{theorem}[Multi-type branching process limit] \label{gwplimit}
   \ \\
   Suppose that the assumptions of Lemma \ref{gwplemma} are satisfied. If $X_0\to Z_0$ in distribution as $N\to\infty$ for some $\nz_0^L$-valued random variable $Z_0$, then, as $N\to\infty$, the forward process $X:=(X_r)_{r\in\nz_0}$, defined via (\ref{forward}), converges in $D_{\nz_0^L}(\nz_0)$ to an $L$-type Galton--Watson branching process $Z:=(Z_r)_{r\in\nz_0}$, whose distribution is characterized as follows. For every $k\in[L]$ and $j=(j_\ell)_{\ell\in[L]}\in\nz_0^L$, the probability that a parent of type $k\in[L]$ produces for each $\ell\in[L]$ exactly $j_\ell$ children of type $\ell$ is given by (\ref{ykdist}).
\end{theorem}
\begin{remark} \rm
   Analogous results for bisexual population models with $K=2$ types are stated in \cite[Theorem 1]{Kaemmerle1991} and \cite[Theorem 2.2 and Corollary 2.4]{Moehle1994}.
\end{remark}
\begin{proof}[Proof of Theorem \ref{gwplimit}]
   Let $n\in\nz_0$ and $j(0),\ldots,j(n)\in\nz_0^L$. For all sufficiently large $N$, by the Markov property and the time-homogeneity of the process $X$,
   \begin{eqnarray*}
      &   & \hspace{-10mm}\pr(X_0=j(0),\ldots,X_n=j(n))\\
      & = & \pr(X_0=j(0))\prod_{r=0}^{n-1}\pr(X_{r+1}=j(r+1)\,|\,X_r=j(r))\\
      & = & \pr(X_0=j(0))\prod_{r=0}^{n-1}\pr(X_1=j(r+1)\,|\,X_0=j(r)).
   \end{eqnarray*}
   By the assumptions and Lemma \ref{gwplemma}, this expression converges as $N\to\infty$ to
   \begin{eqnarray*}
      &   & \hspace{-10mm}
      \pr(Z_0=j(0))\prod_{r=0}^{n-1} \pr\bigg(
      \bigcap_{\ell\in[L]}
         \bigg\{
            \sum_{k\in[L]}Y_{k\ell}^{*j_k(r)}=j_\ell(r+1)
         \bigg\}
      \bigg)\\
%      & = & \pr(Z_0=j(0))\prod_{r=0}^{n-1}\pr(Z_1=j(r+1)\,|\,Z_0=j(r))\\
      & = & \pr(Z_0=j(0))\prod_{r=0}^{n-1}\pr(Z_{r+1}=j(r+1)\,|\,Z_r=j(r))\\
      & = & \pr(Z_0=j(0),\ldots,Z_n=j(n)).
   \end{eqnarray*}
   The convergence of the finite-dimensional distributions is established. For processes with discrete time set $\nz_0$, the convergence of the finite-dimensional distributions is equivalent (see, for example, Billingsley \cite[p.~19]{Billingsley1999}) to the convergence in $D_{\nz_0^L}(\nz_0)$, i.e. to the weak convergence $Q_N\to Q$ as $N\to\infty$, where $Q_N$ and $Q$ denote the image measures of $X$ and $Z$ respectively. Note that $Q_N$ and $Q$ are probability measures on $(\rz^L)^{\nz_0}$ equipped with the Borel $\sigma$-field  generated by the product topology (of pointwise convergence).
\end{proof}
\begin{example}
   Suppose that $\nu_{1,N},\ldots,\nu_{N,N}$ are asymptotically independent (with iid limiting variables $\xi_1,\xi_2,\ldots$) and that $u_{kk}=1$ for all $k\in E$ (absence of mutation). If $L:=K-1\in\nz$ then the assumptions of Lemma \ref{gwplemma} are obviously satisfied and Theorem \ref{gwplimit} is applicable. From (\ref{ykdist}) it follows that (\ref{gwpassumption}) holds with $Y_{k\ell}:=\xi_\ell$ for $k=\ell$ and $Y_{k\ell}:=0$ for $k\ne\ell$.
\end{example}
\begin{example}
   For the Wright--Fisher model, the offspring sizes $\nu_{1,N},\ldots,\nu_{N,N}$ are asymptotically independent, where the limiting variables $\xi_1,\xi_2,\ldots$ are iid copies of a random variable $\xi$ having a Poisson distribution with parameter $1$. From (\ref{ykdist}) and $\me\big((\xi)_{|j|}u_{kK}^{\xi-|j|}\big)=\sum_{n\ge|j|}(n)_{|j|}u_{kK}^{n-|j|}e^{-1}/n!=e^{-1}\sum_{n\ge 0}u_{kK}^n/n!=e^{-\sum_{\ell\in[L]}u_{k\ell}}$ it follows that the random variables $Y_{k\ell}$ are independent with $Y_{k\ell}$ Poisson distributed with parameter $u_{k\ell}$, $k,\ell\in[L]$. Theorem \ref{gwplimit} is applicable provided that $u_{KK}=1$. In this case, for the limiting $L$-type branching process the probability that a parent of type $k\in[L]$ produces for each $\ell\in[L]$ exactly $j_\ell$ children of type $\ell$ is given by $\prod_{\ell\in[L]}\pr(Y_{k\ell}=j_\ell)=\prod_{\ell\in[L]}u_{k\ell}^{j_\ell}e^{-u_{k\ell}}/j_\ell!$.
\end{example}
\begin{remark} \rm
   Without the condition $u_{KK}=1$, Lemma \ref{gwplemma} and, hence, Theorem \ref{gwplimit} in general fail. Consider for example the Wright--Fisher model with $K=2$ types. Fix $i\in\nz_0$. For all $N\ge i$, conditional on $X_1(0)=i$, the random variable $X_1(1)$ has a binomial distribution with parameter $N$ and $\pi_1:=(u_{11}i+u_{21}(N-i))/N$. This binomial distribution does not weakly converge as $N\to\infty$ as long as $u_{21}>0$, since $N\pi_1=u_{11}i+u_{21}(N-i)\to\infty$ in this case. Thus, (\ref{gwpassumption}) does not hold if $u_{22}<1$, so Lemma \ref{gwplemma} and Theorem \ref{gwplimit} are not applicable.
\end{remark}

\subsection{Some backward results} \label{somebackward}

For $N\in\nz$ define $S_N(E):=\{i=(i_k)_{k\in E}\in\nz_0^E:\sum_{k\in E}i_k\le N\}$. Fix $i=(i_k)_{k\in E}\in S_N(E)$. Suppose that one has taken in some generation $r\in\gz$ a sample of $|i|:=\sum_{k\in E}i_k$ individuals, where $i_k$ of these individuals are of type $k$, $k\in E$. For $j=(j_k)_{k\in E}\in S_N(E)$ let $A_{ij}$ denote the event that, for each $k\in E$, the $i_k$ individuals of type $k$ have exactly $j_k$ parents. Define $p_{ij}:=\pr(A_{ij})$ and $P:=(p_{ij})_{i,j\in S_N(E)}$. Clearly, all quantities $A_{ij}=A_{ij}(U)$, $p_{ij}=p_{ij}(U)$ and $P=P(U)$ depend on the mutation matrix $U$. For the model without mutation, that is, for $U=I$ (identity matrix), it is known (see, for example, \cite[Proposition 1]{Moehle2010}) that $P^{\rm rep}:=P(I)$ has entries
\begin{eqnarray}
   p_{ij}^{\rm rep}
   & = & \frac{(N-|i|)!\prod_{k\in E}i_k!}{(N-|j|)!\prod_{k\in E}j_k!}
         \sum_m \me\bigg(\prod_{s=1}^{|j|}\binom{\nu_{s,N}}{m_s}\bigg)\nonumber\\
   & = & \frac{\prod_{k\in E}i_k!}{\prod_{k\in E}j_k!}
         \sum_m \frac{\Phi_{|j|}(m_1,\ldots,m_{|j|})}{m_1!\cdots m_{|j|}!},
         \label{backtransnomut}
\end{eqnarray}
where $|i|:=\sum_{k\in E}i_k$, $|j|:=\sum_{k\in E}j_k$ and the sum extends over all $m=(m_1,\ldots,m_{|j|})\in\nz^{|j|}$ satisfying $m_1+\cdots+m_{j_1}=i_1$,
$m_{j_1+1}+\cdots+m_{j_1+j_2}=i_2,\ldots,m_{j_1+\cdots+j_{K-1}+1}+\cdots+m_{|j|}=i_K$.

For $k\in E$ let $e_k$ denote the $k$-th unit vector in $\rz^E$. From (\ref{backtransnomut}) it follows that, for all $N\in\nz\setminus\{1\}$,
$p_{2e_k,e_k}(I)=\Phi_1(2)=c_N$ does not depend on $k\in E$ and coincides with the coalescence probability (\ref{coalprob}).

In order to obtain formulas for $p_{ij}=p_{ij}(U)$ for general mutation matrix $U$ we proceed as follows. Let $P^{\rm mut}:=(p_{ij}^{\rm mut})_{i,j\in S_N(E)}$ denote the matrix with entries
\begin{equation} \label{bij}
   p_{ij}^{\rm mut}\ :=\ \sum_M \prod_{k\in E}\bigg( i_k!\prod_{\ell\in E}\frac{u_{\ell k}^{m_{k\ell}}}{m_{k\ell}!}\bigg),
   \quad i,j\in S_N(E),
\end{equation}
where the sum extends over all matrices $M=(m_{k\ell})_{k,\ell\in E}\in\nz_0^{E\times E}$ having row sums $\sum_{\ell\in E}m_{k\ell}=i_k$, $k\in E$, and column sums $\sum_{k\in E}m_{k\ell}=j_\ell$, $\ell\in E$. The matrix $P^{\rm mut}$ has a similar structure as the matrix $\Pi^{\rm mut}$ with entries (\ref{aij2}), but note that $u_{k\ell}$ is replaced by its transpose $u_{\ell k}$ and that $i$ and $j$ belong to $S_N(E)$ instead of $\Delta_N(E)$. For general mutation matrix $U$ the backward probabilities $p_{ij}=p_{ij}(U)$ can be calculated using the following result.
\begin{lemma}
   The backward matrix $P=(p_{ij})_{i,j\in S_N(E)}$ (for the model with general mutation matrix $U$) is given by
   \[
   P\ =\ P^{\rm mut}P^{\rm rep},
   \]
   where $P^{\rm mut}=(p_{ij}^{\rm mut})_{i,j\in S_N(E)}$ is the matrix with entries (\ref{bij}) and $P^{\rm rep}=P(I)$ is the backward matrix for the multi-type Cannings without mutation having entries (\ref{backtransnomut}).
\end{lemma}
\begin{proof}
   Fix $i=(i_k)_{k\in E},j=(j_k)_{k\in E}\in S_N(E)$. Recall that the model is defined forward in time by a reproductive step (involving the offspring sizes $\nu_1,\ldots,\nu_N$) followed by a mutational step (involving the mutation matrix $U$). Backward in time we therefore first have to take into account the mutational step and then the reproductive step. Having this factorization in mind it follows that
   \begin{equation} \label{local1}
      p_{ij}(U)
      \ =\ \sum_M \bigg(
          \prod_{k\in E}\bigg(
             i_k!\prod_{\ell\in E}\frac{u_{\ell k}^{m_{k\ell}}}{m_{k\ell}!}
          \bigg)
          \bigg)p_{dj}(I),
   \end{equation}
   where the sum extends over all matrices $M=(m_{k\ell})_{k,\ell\in E}$ having row sums $\sum_{\ell\in E}m_{k\ell}=i_k$ for all $k\in E$ and $d:=(d_\ell)_{\ell\in E}$ is defined via $d_\ell:=\sum_{k\in E}m_{k\ell}$ for all $\ell\in E$. Note that $|d|:=\sum_{\ell\in E}d_\ell=\sum_{k,\ell\in E} m_{k\ell}=\sum_{k\in E}i_k=:|i|$. In particular, $d\in S_N(E)$. Now split the sum in (\ref{local1}) into
   \begin{equation} \label{local2}
      p_{ij}(U)\ =\ \sum_{d\in S_N(E)}
      \bigg(
         \sum_M \prod_{k\in E}
         \bigg(
            i_k!\prod_{\ell\in E}\frac{u_{\ell k}^{m_{k\ell}}}{m_{k\ell}!}
         \bigg)
      \bigg)p_{dj}(I),
   \end{equation}
   where $\sum_M$ extends over all matrices $M=(m_{k\ell})_{k,\ell\in E}$ satisfying $\sum_{\ell\in E}m_{k\ell}=i_k$ for all $k\in E$ and $\sum_{k\in E}m_{k\ell}=d_\ell$ for all $\ell\in E$. Note that the sum $\sum_M$ in (\ref{local2}) is empty if $|d|\ne |i|$. The right-hand side in (\ref{local2}) is equal to $\sum_{d\in S_N(E)}p_{id}^{\rm mut}p_{dj}^{\rm rep}=(P^{\rm mut}P^{\rm rep})_{ij}$. Thus, $P=P^{\rm mut}P^{\rm rep}$.
\end{proof}
\begin{remark} \rm \label{open}
   The matrix $P$ is in general not stochastic, as already mentioned in the remarks on page 715 of \cite{Moehle2010}. It is hence not straightforward to define a proper ancestral process in the same way as for the model studied in Section \ref{first}.
\end{remark}

\subsection{Multi-type Kimura model} \label{kimuramodel}
Let $c\in\nz$. The Cannings population model, where the vector $\nu_N:=(\nu_{1,N},\ldots,\nu_{N,N})$ of offspring sizes has the symmetric multi-hypergeometric distribution
\[
\pr(\nu=j)\ =\ \frac{1}{\binom{cN}{N}}\prod_{i\in[N]}\binom{c}{j_i}
\]
for $j=(j_i)_{i\in[N]}\in\nz_0^N$ with $\sum_{i\in[N]}j_i=N$, is called the Kimura model with parameter $c$; see \cite[p.~636]{Gladstien1978} or \cite{Kimura1957}. Since $D=(D_k)_{k\in E}$ has a multi-hypergeometric distribution with parameters $N$, $cN$ and $ci_k$, $k\in E$, i.e., $\pr(D=j)=\binom{cN}{N}^{-1}\prod_{k\in E}\binom{ci_k}{j_k}$, $i,j\in \Delta_N(E)$, it follows from (\ref{trans}) that the forward process has transition probabilities
\begin{equation} \label{kimura}
   \pi_{ij}
   \ =\ \frac{1}{\binom{cN}{N}}\sum_M\prod_{k\in E}
   \bigg(
      \binom{ci_k}{j_k}d_k!
      \prod_{\ell\in E}\frac{u_{k\ell}^{m_{k\ell}}}{m_{k\ell}!}
   \bigg),
\end{equation}
$i=(i_k)_{k\in E},j=(j_k)_{k\in E}\in\Delta_N(E)$, where the sum extends over all matrices $M=(m_{k\ell})_{k,\ell\in E}$ having column sums $\sum_{k\in E}m_{k\ell}=j_\ell$, $\ell\in E$, and $d_k:=\sum_{\ell\in E}m_{k\ell}$, $k\in E$. Eq.~(\ref{kimura}) does not seem to simplify much further.

It is easily checked that the offspring sizes $\nu_1,\ldots,\nu_N$ are asymptotically independent with $\nu_1\to\xi$ in distribution as $N\to\infty$, where $\xi$ has a binomial distribution with parameters $c$ and $1/c$.

If $L:=K-1\in\nz$ and $u_{KK}=1$, then Theorem \ref{gwplimit} is applicable. For $n\in\nz_0$ and $x\in\rz$,
\begin{eqnarray*}
   &   & \hspace{-12mm}\me\big((\xi)_nx^{\xi-n}\big)
   \ = \ \sum_{m=n}^c (m)_nx^{m-n}\binom{c}{m}\bigg(\frac{1}{c}\bigg)^m\bigg(1-\frac{1}{c}\bigg)^{c-m}\\
   & = & (c)_n \bigg(\frac{1}{c}\bigg)^n
         \sum_{k=0}^{c-n}\binom{c-n}{k}
         \bigg(\frac{x}{c}\bigg)^k\bigg(1-\frac{1}{c}\bigg)^{c-n-k}\\
   & = & (c)_n \bigg(\frac{1}{c}\bigg)^n\bigg(1-\frac{1-x}{c}\bigg)^{c-n}.
\end{eqnarray*}
Thus, for the limiting $L$-type Galton--Watson branching process the probability (\ref{ykdist}) that an individual of type $k\in[L]$ produces for each $\ell\in[L]$ exactly $j_\ell$ children of type $\ell$ is
\[
\pr(Y_k=j)
\ =\ (c)_{|j|}\bigg(\frac{1}{c}\bigg)^{|j|}\bigg(1-\frac{u_k}{c}\bigg)^{c-|j|}
\prod_{l\in[L]}\frac{u_{k\ell}^{j_\ell}}{j_\ell!},
\]
$j=(j_\ell)_{\ell\in[L]}\in\nz_0^L$, where $|j|:=\sum_{\ell\in[L]} j_\ell$ and $u_k:=\sum_{\ell\in[L]}u_{k\ell}$. From Lemma \ref{facmom} it follows that $Y_k$ has joint descending factorial moments
\[
\me\bigg(\prod_{\ell\in[L]}(Y_{k\ell})_{m_\ell}\bigg)
\ =\ (c)_{|m|}\bigg(\frac{1}{c}\bigg)^{|m|}\prod_{\ell\in[L]}u_{k\ell}^{m_\ell},
\quad (m_\ell)\in\nz_0^L,
\]
where $|m|:=\sum_{\ell\in[L]} m_\ell$. Note that $\me(\xi)=1$ and $\me((\xi)_2)=(c)_2(1/c)^2=1-1/c$. For $\ell_1,\ell_2\in[L]$ with $\ell_1\ne\ell_2$ it follows from (\ref{cov}) that ${\rm Cov}(Y_{k\ell_1},Y_{k\ell_2})=-u_{k\ell_1}u_{k\ell_2}/c\le 0$. If the mutation probabilities are strictly positive, then $Y_{k\ell_1}$ and $Y_{k\ell_2}$ are negatively correlated.

\section{Discussion and open problems} \label{discussion}

In Section \ref{first} we have analyzed the ancestry of a multi-type Cannings model with fixed subpopulation sizes and mutation leading to limiting coalescent processes with mutation (see Theorem \ref{main}) enjoying the exchangeability and consistency property.

We have also studied (see Section \ref{second}) a different but closely related Cannings model with constant total population size but variable subpopulation sizes with an emphasis on its forward structure. Under certain conditions its forward structure can be approximated by a limiting multi-type branching process (Theorem \ref{gwplimit}). However, questions concerning its ancestral structure (see Remark \ref{open}) and duality results linking its forward and backward structure, including algebraic approaches to duality, remain open.

Particular classes of multi-type Cannings models have not been discussed in this paper. Schweinsberg~\cite{Schweinsberg2003} studies the ancestry of a class of single-type Cannings model obtained via sampling without replacement from a supercritical branching process. Huillet et al.~\cite{HuilletMoehle2022} study analog single-type models based on a sampling with replacement strategy. We leave the study of multi-type versions of the models of \cite{Schweinsberg2003} and \cite{HuilletMoehle2022} for future work.

%\begin{acknowledgement}
%   The author thanks two anonymous referees for their constructive
%   comments leading to a significant improvement of the article.
%\end{acknowledgement}

%\vfill\eject

\footnotesize

%%%%%%%%%%%%%%%%%%%%%%%%%%%%%%%%%%%%%%
% Bibliography either via bibtex ... %
%%%%%%%%%%%%%%%%%%%%%%%%%%%%%%%%%%%%%%

% Use one of the following bibliography styles:
% abbrv, acm, alpha, apalike, ieeetr, plain, siam, unsrt
\bibliographystyle{plain}
%
% Use your favorite bib-file, for example
%\bibliography{frame}                     % the local bib-file frame.bib or
%\bibliography{c:/biblio/fulllist}       % the master bib-file fulllist.bib or
%\bibliography{/scratch/biblio/fulllist} % the master bib-file at university

%%%%%%%%%%%%%%%%%%%%%%%%%%%%%%%%%%%%%
% or bibliography included directly %
%%%%%%%%%%%%%%%%%%%%%%%%%%%%%%%%%%%%%

\end{document}